\documentclass[twoside,english,3p]{elsarticle}
\usepackage[T1]{fontenc}
\usepackage[latin9]{inputenc}
\usepackage{geometry}
\usepackage{color}
\usepackage{amsmath}
\usepackage{amsthm}
\usepackage{amssymb}
\usepackage{graphicx}

\makeatletter
\theoremstyle{plain}
\newtheorem{thm}{\protect\theoremname}[section]
\theoremstyle{definition}
\newtheorem{example}[thm]{\protect\examplename}
\theoremstyle{remark}
\newtheorem{rem}[thm]{\protect\remarkname}
\theoremstyle{plain}
\newtheorem{lem}[thm]{\protect\lemmaname}

\journal{JDE}

\usepackage{babel}

\makeatother

\usepackage{babel}
\providecommand{\examplename}{Example}
\providecommand{\lemmaname}{Lemma}
\providecommand{\remarkname}{Remark}
\providecommand{\theoremname}{Theorem}

\makeatletter
\def\ps@pprintTitle{%
  \let\@oddhead\@empty
  \let\@evenhead\@empty
  \def\@oddfoot{\reset@font\hfil\thepage\hfil}
  \let\@evenfoot\@oddfoot
}
\makeatother

\begin{document}
\begin{frontmatter}{}

\title{A class of third order partial differential equations describing
spherical or pseudospherical surfaces}

\author[rvt]{Diego Catalano Ferraioli\fnref{fn1}\corref{cor1}}

\ead{diego.catalano@ufba.br}

\author[focal]{Tarc\' isio Castro Silva\fnref{fn2}}

\fntext[fn1]{Partially supported by CNPq, grant 422906/2016-6.}

\fntext[fn2]{Partially supported by FEMAT, grant 01/2022 and by CNPq, grant 422906/2016-6.}

\cortext[cor1]{Corresponding author}

\address[rvt]{Department of Mathematics, Universidade Federal da Bahia, Campus
de Ondina, Av. Adhemar de Barros, S/N, Ondina - CEP 40.170.110 - Salvador,
BA - Brazil, e-mail: diego.catalano@ufba.br.}

\address[focal]{Department of Mathematics, Universidade de Brasilia, Brasilia DF
70910-900, Brazil, e-mail: tarcisio@mat.unb.br. }
\begin{abstract}
Third order equations, which describe spherical surfaces (ss) or pseudospherical surfaces
(pss), of the form 
\[
\nu\,z_{t}-\lambda\,z_{xxt}=A(z,z_{x},z_{xx})\,z_{xxx}+B(z,z_{x},z_{xx})
\]
with $\nu$, $\lambda$ $\in$ $\mathbb{R}$, $\nu^2+\lambda^2\neq 0$, are considered. These equations are equivalent to the structure
equations of a metric with Gaussian curvature $K=1$ or $K=-1$, respectively.
Alternatively they can be seen as the compatibility condition of an
associated $\mathfrak{su}(2)$-valued or $\mathfrak{sl}(2,\mathbb{R})$-valued
linear problem, also referred to as a zero curvature representation.
Under certain assumptions we obtain an explicit classification for
equations of the considered form that describe ss or pss, in terms
of some arbitrary differentiable functions. Several examples of such
equations, which describe also a number of already known equations,
are provided by suitably choosing the arbitrary functions. 
\end{abstract}
\end{frontmatter}{}

\section{Introduction}

Nonlinear partial differential equations describing spherical surfaces
(\textbf{ss}) or pseudospherical surfaces (\textbf{pss}) are characterized
by the fact that their generic solutions provide metrics, on nonempty
open subset of $\mathbb{R}^{2}$, with Gaussian curvature $K=1$ or
$K=-1$, respectively. 

The first appearance in the literature of an equation that describes
\textbf{pss} was in 1862, with a paper \citep{EB} by Edmund Bour
who was the first to find that, in Darboux asymptotic coordinates,
the Gauss-Mainardi-Codazzi equations of a pseudospherical surface
reduce to the equation $z_{xt}=\sin\,z$, now better known as the
sine-Gordon (SG) equation. Subsequently, this equation aroused considerable
interest, due to the advent of B\"acklund's theory of transformations
and the results obtained by Bianchi on a nonlinear superposition principle
for its solutions. However the general theory of equations describing
\textbf{pss} was born with a paper by S.S. Chern and K. Tenenblat
\citep{CT}, which was motivated by the early observation \citep{Sas}
that \textquotedblleft all the soliton equations in $1+1$ dimensions
that can be solved by the AKNS $2\times2$ inverse scattering method
(for example, the sine-Gordon, KdV or modified KdV equations) ...
describe pseudospherical surfaces\textquotedblright . Thus began a
systematic study of these equations which allowed not only to better
understand their general properties, but also to obtain various classification
results. Subsequently in \citep{DT}, in analogy with \citep{CT},
the general notion of equation describing \textbf{ss} was also introduced.
For an overview of past researches on equations describing \textbf{ss}
and \textbf{pss} the reader is referred to \citep{BT,BRT,CamT,Tarcisio-Keti,Tarcisio-Niky,Catalano-Tarcisio-Keti,Catalano-Silva,Diego-Luis,FT,CaT,CT1,G,JT,KT,KelKeti,R,RT,T}
and also \citep{Re0,Re1,Re2,Re3,Re4}.

In this paper, we are interested in studying third order quasilinear
partial differential equations of the form

\begin{equation}
\nu\,z_{t}-\lambda\,z_{xxt}=A(z,z_{x},z_{xx})\,z_{xxx}+B(z,z_{x},z_{xx})\label{eq_z_x..}
\end{equation}
with $\nu,\,\lambda\in\mathbb{R}$, $\nu^{2}+\lambda^{2}\neq0$ and
$A^{2}+B^{2}\neq0$, that describe spherical or pseudospherical surfaces.

To this end, depending on whether $\nu\neq0$ or $\nu=0$, equations
\eqref{eq_z_x..} will be distinguished into two classes: 

\begin{description}
\item [$\bullet$] $\quad$ $z_{t}-\lambda\,z_{xxt}=A(z,z_{x},z_{xx})\,z_{xxx}+B(z,z_{x},z_{xx})$,
with $\lambda\in\mathbb{R}$ and $A^{2}+B^{2}\neq0$; 
\item [$\bullet$] $\quad$ $z_{xxt}=A(z,z_{x},z_{xx})\,z_{xxx}+B(z,z_{x},z_{xx})$,
with $A^{2}+B^{2}\neq0$. 
\end{description}
Indeed, one can always reduce an equation of the form \eqref{eq_z_x..}
to one of the form \textbf{(a)} or \textbf{(b)} by suitably rearranging
arbitrary functions $A,B$ and the constant $\lambda$.

Then in order to deal with classification problem we will assume the
following auxiliary conditions: 
\begin{description}
\item [{(i)}] $f_{11}=\eta$ is constant; 
\item [{(ii)}] $f_{21}$ and $f_{31}$ are linear non homogeneous functions
of $z$ and $z_{xx}$ with constant coefficients. 
\end{description}
The paper is organized as follow. In Section \ref{Prelim}, we collect
some preliminaries on differential equations that describe pseudospherical
or spherical surfaces. Then in Section \ref{sec3} we state our main
results and provide some explicit examples. The complete proofs are
given in Sections \ref{sec4} and \ref{sec4*}.

\medskip{}

\section{Preliminaries}

\label{Prelim}

For the reader's convenience we collect here some basic facts, and
notations used throughout the paper, from the theory of equations
describing pseudospherical or spherical surfaces.

If $\left(M,\,g\right)$ is a 2-dimensional Riemannian manifold and
$\left\{ \omega_{1},\omega_{2}\right\} $ is a co-frame, dual to an
orthonormal frame $\left\{ e_{1},e_{2}\right\} $, then $g=\omega_{1}^{2}+\omega_{2}^{2}$
and $\omega_{i}$ satisfy the structure equations: $d\omega_{1}=\omega_{3}\wedge\omega_{2}$
and $d\omega_{2}=\omega_{1}\wedge\omega_{3}$, where $\omega_{3}$
denotes the connection form defined as $\omega_{3}(e_{i})=d\omega_{i}(e_{1},e_{2})$.
The Gaussian curvature of $M$ is the function $K$ such that $d\omega_{3}=-K\omega_{1}\wedge\omega_{2}$.

Now, a $k$-th order differential equation $\mathcal{E}$, for a scalar
or vector real-valued function $z\left(x,t\right)$, \emph{describes
pseudospherical surfaces }\textbf{(pss)}\emph{, or spherical surfaces
}\textbf{(ss)}, if it is equivalent to the structure equations of
a surface with Gaussian curvature $K=-\delta$, with $\delta=1$ or
$\delta=-1$, respectively, i.e., 
\begin{equation}
\begin{array}{l}
d\omega_{1}=\omega_{3}\wedge\omega_{2},\quad d\omega_{2}=\omega_{1}\wedge\omega_{3},\quad d\omega_{3}=\delta\omega_{1}\wedge\omega_{2},\end{array}\label{struttura}
\end{equation}
where $\left\{ \omega_{1},\omega_{2},\omega_{3}\right\} $ are $1$-forms
\begin{equation}
\begin{array}{l}
\omega_{1}=f_{11}dx+f_{12}dt,\quad\omega_{2}=f_{21}dx+f_{22}dt,\quad\omega_{3}=f_{31}dx+f_{32}dt,\end{array}\label{eq:forms}
\end{equation}

\noindent such that $\omega_{1}\wedge\omega_{2}\neq0$, i.e., 
\begin{equation}
f_{11}f_{22}-f_{12}f_{21}\neq0,\label{eq:nondeg_cond}
\end{equation}
and $f_{ij}$ are functions of $x$, $t$, $z(x,t)$ and derivatives
of $z(x,t)$ with respect to $x$ and $t$.

Notice that, according to the definition, given a solution $z(x,t)$
of a \textbf{pss} (or \textbf{ss}) equation $\mathcal{E}$, we consider
an open connected set $U\subset\mathbb{R}^{2}$, contained in the
domain of $z(x,t)$, where the restriction of $\omega_{1}\wedge\omega_{2}$
to $z$ is everywhere nonzero on $U$. Such an open set $U$ exists
for generic solutions $z$. Thus, for generic solutions $z$ of a
\textbf{pss} (or \textbf{ss}) $\mathcal{E}$, the restriction $I[z]$
of $I=\omega_{1}^{2}+\omega_{2}^{2}$ to $z$ defines a Riemannian
metric $I[z]$, on $U$, with Gaussian curvature $K=-\delta$. It
is in this sense that one can say that a \textbf{pss} (\textbf{ss},
resp.) describes, \textbf{pss} (\textbf{ss}, resp.). This is an intrinsic
geometric property of a Riemannian metric (not immersed in an ambient
space).

A classical example of equation describing \textbf{pss} is the sine-Gordon
(SG) equation $z_{xt}=sin\left(z\right)$, which corresponds to 
\[
\begin{array}{l}
\omega_{1}=\frac{1}{\eta}sin\left(z\right)\,dt,\qquad\omega_{2}=\eta\,dx+\frac{1}{\eta}cos\left(z\right)\,dt,\qquad\omega_{3}=z_{x}\,dx,\end{array}
\]
with $\eta\in\mathbb{R}-\{0\}$ \citep{CT}. Observe that the solution
$z(x,t)\equiv0$ is not generic, in the sense that there is no open
set $U$, where $\omega_{1}\wedge\omega_{2}\neq0$, since $\omega_{1}(z)\equiv0$.

On the other hand, a well known example of equation describing \textbf{ss}
is the nonlinear Schrodinger system $NLS^{+}$ 
\[
\left\{ \begin{array}{l}
u_{t}+v_{xx}+2\left(u^{2}+v^{2}\right)u=0,\vspace{4pt}\\
-v_{t}+u_{xx}+2\left(u^{2}+v^{2}\right)v=0,
\end{array}\right.
\]
for the vector-valued function $\mathbf{z}(x,t)=(u(x,t),v(x,t))$.
This equation corresponds to 
\[
\begin{array}{l}
\omega_{1}=2vdx+\left(-4\eta v+2u_{x}\right)dt,\vspace{4pt}\\
\omega_{2}=2\eta dx+\left(-4\eta^{2}+2\left(u^{2}+v^{2}\right)\right)dt,\vspace{4pt}\\
\omega_{3}=-2udx+\left(2\eta u+2v_{x}\right)dt,
\end{array}
\]
with $\eta\in\mathbb{R}$ \citep{DT}.

An interesting property of equations describing {\bf pss} is that one may obtain an infinite hierarchy of conservation laws
for such an equation, as a consequence of the following geometric
properties. Assume that a differential equation ${\cal E}$ for $z(x,t)$
describes pseudospherical surfaces, with associated 1-forms $\omega_{1},\,\omega_{2},\,\omega_{3}$,
then $z$ is a solution of ${\cal E}$ if and only if $\omega_{3}-d\rho\,\omega_{1}+\sin\rho\,\omega_{1}+\cos\rho\,\omega_{2}=0$
is completely integrable for $\rho$. For each solution $z$ of ${\cal E}$
and corresponding solution $\rho$, the 1-form $\cos\rho\,\omega_{1}-\sin\rho\,\omega_{2}$
is closed (see the proof of Proposition 4.2 in \citep{CT}). Whenever
the 1-forms $\omega_{i}$, $i=1,2,3$ are analytic in a parameter
$\eta$, then $\rho$ is also analytic on the parameter and hence
the conservation laws may be obtained by the closed form written as
a series in $\eta$. Analogous properties hold for equations
describing {\bf ss}.

Equations which describe \textbf{pss}, or \textbf{ss}, can also be
characterized in few alternative ways. For instance, the system of
equations (\ref{struttura}) is equivalent to the integrability condition
\begin{equation}
d\Omega-\Omega\wedge\Omega=0,\label{eq:ZCR_omega}
\end{equation}
of the linear system 
\begin{equation}
dV=\Omega V,\label{struttura-2}
\end{equation}
for an auxiliary differentiable function $V=(v^{1},v^{2})^{T}$, with
$v^{i}=v^{i}\left(x,t\right)$, where $\Omega$ is either the $\mathfrak{sl}\left(2,\mathbb{R}\right)$-valued
1-form 
\[
\Omega=\frac{1}{2}\left(\begin{array}{cc}
\omega_{2} & \omega_{1}-\omega_{3}\\
\omega_{1}+\omega_{3} & -\omega_{2}
\end{array}\right),\qquad\text{when \;}\delta=1,
\]
or the $\mathfrak{su}\left(2\right)$-valued 1-form 
\[
\Omega=\frac{1}{2}\left(\begin{array}{cc}
i\omega_{2} & \omega_{1}+i\omega_{3}\\
-\omega_{1}+i\omega_{3} & -i\omega_{2}
\end{array}\right),\qquad\text{when \;}\delta=-1.
\]

\noindent Hence, for any solution $z=z\left(x,t\right)$ of $\mathcal{E}$,
defined on a domain $U\subset\mathbb{R}^{2}$, $\Omega$ is a Maurer-Cartan
form defining a flat connection on a trivial principal $SL\left(2,\mathbb{R}\right)$-bundle,
or $SU\left(2\right)$-bundle, over $\mathcal{U}$ (see for instance
\citep{DUB-KOM,SHARPE}).

Moreover, by using the matrices $X$ and $T$ such that $\Omega=Xdx+Tdt$
and $V:=\left(v^{1},v^{2}\right)^{T}$, (\ref{struttura-2}) can be
written as the linear problem 
\begin{equation}
\frac{\partial V}{\partial x}=XV,\quad\frac{\partial V}{\partial t}=TV.\label{eq:Prob_lin}
\end{equation}

\noindent It is easy to show that equations (\ref{struttura}) (or
(\ref{eq:ZCR_omega})) are equivalent to the integrability condition
of (\ref{eq:Prob_lin}), namely 
\begin{equation}
D_{t}X-D_{x}T+\left[X,T\right]=0,\label{eq:ZCR}
\end{equation}
where $D_{t}$ and $D_{x}$ are the total derivative operators with
respect to $t$ and $x$, respectively.

In the literature \citep{CraPirRob} 1-form $\Omega$, and sometimes
the pair $\left(X,T\right)$ or even (\ref{eq:ZCR}), is referred
to as an $\mathfrak{sl}\left(2,\mathbb{R}\right)$-valued, or $\mathfrak{su}\left(2\right)$-valued\emph{,
zero-curvature representation }(ZCR) for the equation $\mathcal{E}$.
Moreover, the linear system (\ref{struttura-2}) or (\ref{eq:Prob_lin})
is usually referred to as \emph{the linear problem associated to $\mathcal{E}$.
} It is this linear problem that, in some cases, is used in the construction
of explicit solutions of equations describing \textbf{pss}, or \textbf{ss},
by means of inverse scattering method \citep{AKNS,Beals,Beal-Coif,BRT,GGKM}.

We notice here that, given an $\mathfrak{sl}\left(2,\mathbb{R}\right)$-valued ZCR $(X,T)$ for
an equation $\mathcal{E}$ describing {\bf pss} under a \textit{gauge transformation} $X\rightarrow X^{S}=S\,X\,S^{-1}+D_{x}S\,S^{-1}$, $T\rightarrow T^{S}=S\,T\,S^{-1}+D_{t}S\,S^{-1}$ by an $SL(2,\mathbb{R})$-valued smooth function $S$, the left hand side of (\ref{eq:ZCR}) transforms
to 
\[
D_{t}X-D_{x}T+\left[X,T\right]=S\,\left(D_{t}X^{S}-D_{x}T^{S}+\left[X^{S},T^{S}\right]\right)\,S^{-1},
\]
hence $(X^{S},T^{S})$ is another ZCR for $\mathcal{E}$. Analogously, given an $\mathfrak{su}\left(2\right)$-valued ZCR $(X,T)$ for an equation $\mathcal{E}$ describing {\bf ss},
the pair $(X^{S},T^{S})$ defined by an $SU\left(2\right)$-valued
smooth function $S$ is another ZCR for $\mathcal{E}$.

\section{Main results and some examples}

\label{sec3}

From now on, in order to simplify notations, $z_{1}$, $z_{2}$, ...
will denote partial derivatives $z_{x}$, $z_{xx}$, ... of $z$ with
respect to $x$. Accordingly, equation (\ref{eq_z_x..}) will be rewritten
as 
\begin{equation}
\nu\,z_{t}-\lambda\,z_{2t}=A(z,z_{1},z_{2})\,z_{3}+B(z,z_{1},z_{2})\label{eq}
\end{equation}
where $z_{1t}$, $z_{2t}$, ... denote partial derivatives of $z_{1}$,
$z_{2}$, ... with respect to $t$, and condition \textbf{(ii)} becomes
\textquotedbl$f_{21}$ and $f_{31}$ are linear non homogeneous functions
of $z$ and $z_{2}$ with constant coefficients\textquotedbl{} .

\begin{figure}[h!]
\centering{}\caption{Classification scheme}
\includegraphics[scale=0.45]{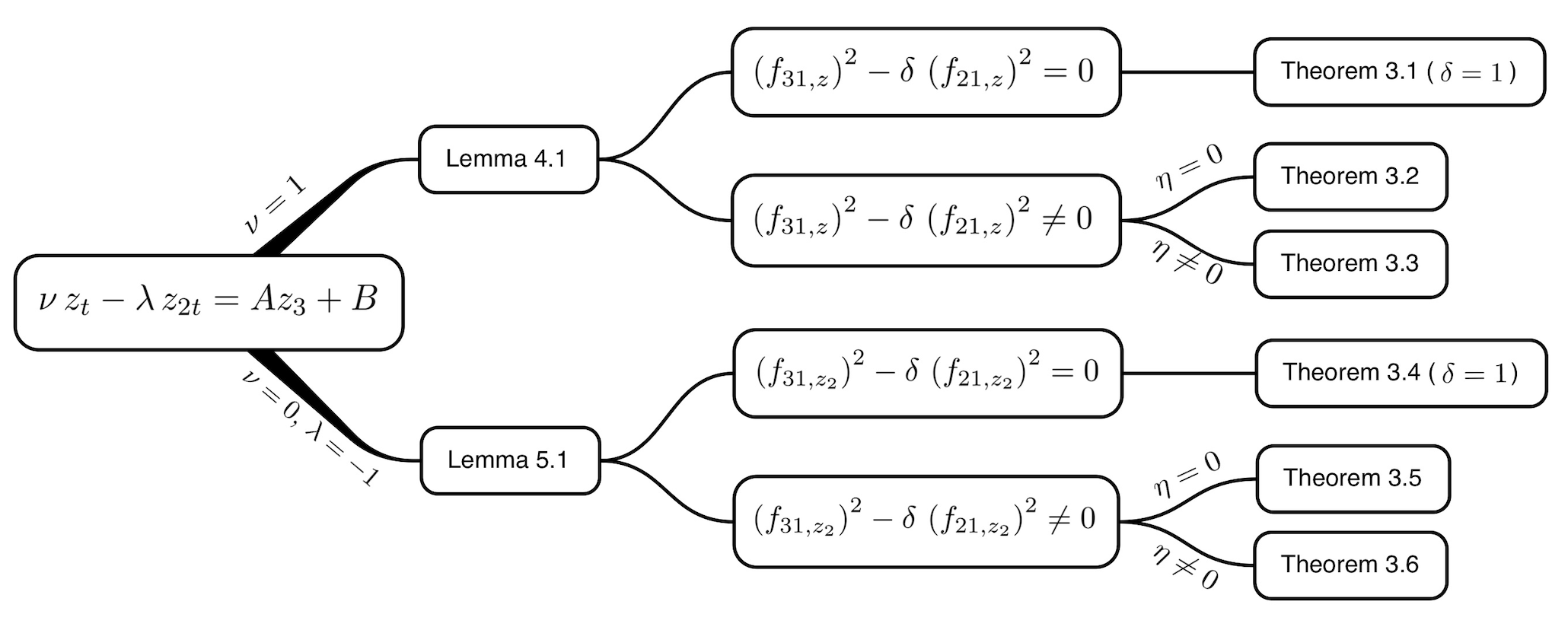} \label{fig:Diagrama1} 
\end{figure}

As already observed in the introduction, depending on whether $\nu\neq0$
or $\nu=0$, these equations will be distinguished into two classes: 
\begin{description}
\item [{(a)}] $\qquad$$z_{t}-\lambda\,z_{2t}=A(z,z_{1},z_{2})\,z_{3}+B(z,z_{1},z_{2})$,
with $\lambda\in\mathbb{R}$ and $A^{2}+B^{2}\neq0$; 
\item [{(b)}] $\qquad$$z_{2t}=A(z,z_{1},z_{2})\,z_{3}+B(z,z_{1},z_{2})$,
with $A^{2}+B^{2}\neq0$. 
\end{description}
The scheme followed in the paper for the classification of equations
(\ref{eq}) that describe \textbf{pss} or \textbf{ss} ($\delta=1$
or $\delta=-1,$ respectively), with corresponding functions $f_{ij}$
satisfying conditions \textbf{(i-ii)}, is summarized in Figure 1.

For the sake of simplicity, since the form of equation (\ref{eq})
is left invariant by linear coordinate transformations 
\[
x\mapsto ax,\quad t\mapsto bt,\quad z\mapsto c\,z+c_{0},\qquad a\,b\,c\neq0,\qquad a,b,c,c_{0}\in\mathbb{R}
\]
that also preserve conditions $\left(f_{31,z}\right)^{2}-\delta\left(f_{21,z}\right)^{2}=0$,
$\left(f_{31,z}\right)^{2}-\delta\left(f_{21,z}\right)^{2}\neq0$
and the functional dependence of $f_{i1}$ under $z$ and $z_{2}$
described by \textbf{(i-ii),} the classification will be given only
up to such coordinate transformations. This choice will simplify noteworthy
the results of the classification, without loss of generality.

In this section we only state the main results of the paper and further
illustrate them by means of some simple examples. The complete proofs
are postponed to Sections \ref{sec4} and \ref{sec4*}.

\subsection{Equations of the form $z_{t}-\lambda\,z_{2t}=A\,z_{3}+B$\label{Sec2.2}}

We collect here our main results on equations of the form 
\begin{equation}
z_{t}-\lambda\,z_{2t}=A\,z_{3}+B\label{eq:C-1}
\end{equation}
that describe \textbf{pss} or \textbf{ss} ($\delta=1$ or $\delta=-1,$
respectively), with corresponding functions $f_{ij}=f_{ij}(z,z_{1},...,z_{k})$
satisfying \textbf{(i-ii).}

The first theorem describes the case when $\{\left(f_{31,z}\right)^{2}-\delta\left(f_{21,z}\right)^{2}=0\}$,
whereas the next two describe the cases $\{\left(f_{31,z}\right)^{2}-\delta\left(f_{21,z}\right)^{2}\neq0,\;\eta=0\}$
and $\{\left(f_{31,z}\right)^{2}-\delta\left(f_{21,z}\right)^{2}\neq0,\;\eta\neq0\}$,
respectively. \medskip{}

\begin{thm}
\label{teo.clas.3.2} A partial differential equation of the form
\eqref{eq:C-1} describes \textbf{pss} or \textbf{ss} ($\delta=1$
or $\delta=-1,$ respectively), with corresponding functions $f_{ij}=f_{ij}(z,z_{1},...,z_{k})$
satisfying \textbf{(i-ii)} with $\left(f_{31,z}\right)^{2}-\delta\left(f_{21,z}\right)^{2}=0$,
if and only if it can be written in one of the following two forms:\\

\noindent \textbf{(I)} 
\[
z_{t}-\lambda\,z_{2t}=D_{x}\psi+a\,\psi+b\,(z-\lambda\,z_{2})
\]

\noindent with $\delta=1$\textbf{ }and associated 1-forms 
\[
\begin{array}{l}
\omega_{1}=\mp\left(a\,dx-b\,dt\right)\,,\vspace{5pt}\\
\omega_{2}=\left(z-\lambda z_{2}\right)\,dx+\psi\,dt\,,\vspace{5pt}\\
\omega_{3}=\pm\,\omega_{2}\,,
\end{array}
\]
where $\psi=\psi(z,z_{1},z_{2})$ is a real differentiable function
and $a,b\in\{0,1\}$, with $a^{2}+b^{2}\neq0$ and $a\,\psi+b\,(z-\lambda\,z_{2})\neq0$;

\vspace*{0.4cm}

\noindent \textbf{(II)} 
\[
z_{t}-\lambda\,z_{2t}=\mp D_{x}^{3}h-D_{x}\left(\left(z-\lambda z_{2}\right)h\right)-\left(z-\lambda z_{2}\right)\,D_{x}\left(h\right)
\]

\noindent with $\delta=1$\textbf{ }and associated 1-forms{\small{}{}{}{}{}{}{}
\[
\begin{array}{l}
\omega_{1}=\eta\,dx+\left(\mp D_{x}h-\eta\,h\right)\,dt\,,\vspace{5pt}\\
\omega_{2}={\displaystyle \frac{1}{\alpha}}\left[\left(z-\lambda z_{2}\pm{\displaystyle \frac{\eta^{2}}{2}}\mp{\displaystyle \frac{\alpha^{2}}{2}}\right)\,dx+\left(\mp D_{x}^{2}h-\eta\,D_{x}h-\left(z-\lambda z_{2}\pm{\displaystyle \frac{\eta^{2}}{2}}\mp{\displaystyle \frac{\alpha^{2}}{2}}\right)h\right)\,dt\right]\,,\vspace{5pt}\\
\omega_{3}=\pm\,\omega_{2}+\alpha\,(dx-h\,dt),
\end{array}
\]
} where $h=h(z)$ is a real differentiable function satisfying $h'\neq0$
and $\eta,\alpha\in\mathbb{R}$, $\alpha\neq0$. 
\end{thm}

\medskip{}

\begin{thm}
\label{teo.clas.3.3} A partial differential equation of the form
\eqref{eq:C-1} describes \textbf{pss} or \textbf{ss} ($\delta=1$
or $\delta=-1,$ respectively), with corresponding functions $f_{ij}=f_{ij}(z,z_{1},...,z_{k})$
satisfying \textbf{(i-ii)} with $\left(f_{31,z}\right)^{2}-\delta\left(f_{21,z}\right)^{2}\neq0$
and $\eta=0$, if and only if it can be written in the following form
\[
z_{t}-\lambda\,z_{2t}=D_{x}^{2}\left(2\,\gamma\,z_{1}\varphi'\right)-D_{x}\left(\left(z-\lambda z_{2}\right)\varphi\right)-2\delta\,r^{2}\,z_{1}\,\varphi'
\]
with $\delta=\pm1$ and associated 1-forms{\small{}{}{}{}{}{}{}
\[
\begin{array}{l}
\omega_{1}={\displaystyle -2r}z_{1}\varphi'\,dt\,,\vspace{5pt}\\
\omega_{2}=\left({\displaystyle \frac{\sigma}{\gamma}}\,(z-\lambda\,z_{2})\pm r\,{\displaystyle \frac{\sqrt{\gamma+\delta\sigma^{2}}}{\gamma}}\right)\,dx+\left[\sigma D_{x}\left(2z_{1}\varphi'\right)+\left(-{\displaystyle \frac{\sigma}{\gamma}}\,\left(z-\lambda z_{2}\right)\mp r{\displaystyle \frac{\sqrt{\gamma+\delta\sigma^{2}}}{\gamma}}\right)\,\varphi\right]\,dt\,,\vspace{5pt}\\
\omega_{3}=\pm\left({\displaystyle \frac{\sqrt{\gamma+\delta\sigma^{2}}}{\gamma}}\,(z-\lambda\,z_{2})\pm{\displaystyle \frac{\delta\,\sigma\,r}{\gamma}}\right)\,dx\pm\left[\sqrt{\gamma+\delta\sigma^{2}}\,D_{x}\left(2z_{1}\varphi'\right)+\left(-{\displaystyle \frac{\sqrt{\gamma+\delta\sigma^{2}}}{\gamma}}\,\left(z-\lambda z_{2}\right)\mp{\displaystyle \frac{\delta\,\sigma\,r}{\gamma}}\right)\,\varphi\right]\,dt\,,
\end{array}
\]
}{\small\par}

\noindent where $\varphi=\varphi(\lambda z_{1}^{2}-z^{2})$ is a differentiable
function such that $\varphi'\neq0$, and $r,\gamma,\sigma\in\mathbb{R}$
with $r\gamma\neq0$ and $\gamma+\delta\sigma^{2}\geq0$. 
\end{thm}

\medskip{}

\begin{thm}
\label{teo.clas.3.5} A partial differential equation of the form
\eqref{eq:C-1} describes \textbf{pss} or \textbf{ss} ($\delta=1$
or $\delta=-1,$ respectively), with corresponding functions $f_{ij}=f_{ij}(z,z_{1},...,z_{k})$
satisfying \textbf{(i-ii)} with $\left(f_{31,z}\right)^{2}-\delta\left(f_{21,z}\right)^{2}\neq0$
and $\eta\neq0$, if and only if it can be written in one of the following
two forms: \\

\textbf{(I)} 
\[
\begin{array}{ll}
z_{t}-\lambda\,z_{2t}= & D_{x}^{2}\left(\varphi\right)-D_{x}\left(\left(z-\lambda z_{2}\right)\varphi\right)\end{array}
\]
\begin{flushleft}
with $\delta=1$ and associated 1-forms 
\[
\begin{array}{l}
\omega_{1}=\eta\,\left(dx-\varphi\,dt\right)\,,\vspace{5pt}\\
\\
\omega_{2}=\left({\displaystyle -\rho}\,(z-\lambda\,z_{2})\pm\eta\,\sqrt{\rho^{2}-1}\right)\,dx\\
\hspace*{3cm}+\left[{\displaystyle -\rho}D_{x}\left(\varphi\right)+\left(\rho\,\left(z-\lambda z_{2}\right)\mp\eta\,\sqrt{\rho^{2}-1}\right)\,\varphi\right]\,dt\,,\vspace{5pt}\\
\\
\omega_{3}=\left({\displaystyle \mp\sqrt{\rho^{2}-1}}\,(z-\lambda\,z_{2})+\rho\,\eta\right)\,dx\\
\hspace*{3cm}+\left[{\displaystyle \mp\sqrt{\rho^{2}-1}}\,D_{x}\left(\varphi\right)+\left(\pm{\displaystyle \sqrt{\rho^{2}-1}}\,\left(z-\lambda z_{2}\right)-\eta\,\rho\right)\,\varphi\right]\,dt\,,
\end{array}
\]
\vspace*{0.3cm}
 where $\varphi=\varphi(z,z_{1})$ is a differentiable function such
that $D_{x}\varphi\neq0$, $\eta,\rho\in\mathbb{R}-\{0\}$ and $\rho^{2}-1\geq0$; 
\par\end{flushleft}
\noindent \begin{flushleft}
\vspace*{0.5cm}
 \textbf{(II)} 
\[
\begin{array}{ll}
z_{t}-\lambda\,z_{2t}= & \frac{1}{\eta}D_{x}^{2}\left(-{\displaystyle \frac{2\eta\gamma z_{1}\ell'+r\,\ell}{\gamma\eta^{2}+r^{2}}}\right)+\frac{1}{\eta}D_{x}\left({\displaystyle \frac{\left(-2r\,z_{1}\ell'+\eta\ell\right)}{\gamma\eta^{2}+r^{2}}\,\left(z-\lambda z_{2}\right)}\right)+2\delta z_{1}\ell'\end{array}
\]
\par\end{flushleft}
\begin{flushleft}
with $\delta=\pm1$ and associated 1-forms {\small{}{}{}{}{}{}{}{}{}{}
\[
\begin{array}{l}
\omega_{1}=\eta\,dx+\left({\displaystyle \frac{\left(-2\,r\,z_{1}\ell'+\eta\ell\right)}{\gamma\eta^{2}+r^{2}}}\right)\,dt\,,\vspace{8pt}\\
\omega_{2}=\left({\displaystyle \frac{{\displaystyle \sigma}}{\gamma}}\,(z-\lambda\,z_{2})\mp\,r\,{\displaystyle \frac{\sqrt{\gamma+\delta\sigma^{2}}}{\gamma}}\right)\,dx\vspace{3pt}\\
\qquad+\left(-{\displaystyle \frac{{\displaystyle \sigma}}{\eta\,\gamma}}D_{x}\left({\displaystyle \frac{2\,\eta\,\gamma\,z_{1}\ell'+r\,\ell}{\gamma\eta^{2}+r^{2}}}\right)+{\displaystyle \frac{\sigma}{\eta\,\gamma}}\,{\displaystyle \frac{\left(-2\,r\,z_{1}\ell'+\eta\,\ell\right)}{\gamma\eta^{2}+r^{2}}}\left(z-\lambda z_{2}\right)\mp{\displaystyle \frac{\sqrt{\gamma+\delta\sigma^{2}}}{\gamma}}\,\left({\displaystyle \frac{2\,\eta\,\gamma\,z_{1}\ell'+r\,\ell}{\gamma\eta^{2}+r^{2}}}\right)\right)\,dt\,,\vspace{8pt}\\
\omega_{3}=\pm\left({\displaystyle \frac{\sqrt{\gamma+\delta\sigma^{2}}}{\gamma}}\,(z-\lambda\,z_{2})\mp{\displaystyle \frac{\delta\,\sigma\,r}{\gamma}}\right)\,dx\vspace{3pt}\\
\qquad\pm\left(-{\displaystyle \frac{\sqrt{\gamma+\delta\sigma^{2}}}{\eta\,\gamma}}\,D_{x}\left({\displaystyle \frac{2\eta\gamma z_{1}\ell'+r\,\ell}{\gamma\eta^{2}+r^{2}}}\right)+{\displaystyle \frac{\sqrt{\gamma+\delta\sigma^{2}}}{\eta\,\gamma}}\,{\displaystyle \frac{\left(-2r\,z_{1}\ell'+\eta\ell\right)}{\gamma\eta^{2}+r^{2}}\,\left(z-\lambda z_{2}\right)}\mp{\displaystyle \frac{\delta\,\sigma}{\gamma}}\,\left({\displaystyle \frac{2\eta\gamma z_{1}\ell'+r\,\ell}{\gamma\eta^{2}+r^{2}}}\right)\right)\,dt
\end{array}
\]
}{\small\par}
\par\end{flushleft}
\begin{flushleft}
 
\par\end{flushleft}
\begin{flushleft}
\vspace*{0.3cm}
 where $\eta,r,\gamma,\sigma\in\mathbb{R}$ satisfy $\gamma+\delta\sigma^{2}\geq0$
and $\left(\gamma\eta^{2}+r^{2}\right)\eta\,\gamma\neq0$, whereas
$\ell=\ell(\lambda z_{1}^{2}-z^{2})$ is a differentiable function
such that $\ell'\neq0$. 
\par\end{flushleft}

\end{thm}

\medskip{}

A remark is in order here. By performing a transformation $z\mapsto z/\nu$,
$\nu\neq0$, one can easily get from Theorems \ref{teo.clas.3.2}-\ref{teo.clas.3.5}
analogous theorems for equations \eqref{eq} with $\nu\neq0$.

\subsection{Equations of the form $z_{2t}=A\,z_{3}+B$}

We collect here our main results on equations of the form 
\begin{equation}
z_{2t}=A\,z_{3}+B\label{eq:C-2}
\end{equation}
that describe \textbf{pss} or \textbf{ss} ($\delta=1$ or $\delta=-1,$
respectively), with corresponding functions $f_{ij}=f_{ij}(z,z_{1},...,z_{k})$
satisfying \textbf{(i-ii).}

The first theorem describes the case when $\{\left(f_{31,z_{2}}\right)^{2}-\delta\left(f_{21,z_{2}}\right)^{2}=0\}$,
whereas the next two describe the cases $\{\left(f_{31,z_{2}}\right)^{2}-\delta\left(f_{21,z_{2}}\right)^{2}\neq0,\;\eta=0\}$
and $\{\left(f_{31,z_{2}}\right)^{2}-\delta\left(f_{21,z_{2}}\right)^{2}\neq0,\;\eta\neq0\}$,
respectively.

\medskip{}

\begin{thm}
\label{teo.clas.3.2*} A partial differential equation of the form
\eqref{eq:C-2} describes \textbf{pss} or \textbf{ss} ($\delta=1$
or $\delta=-1,$ respectively), with corresponding functions $f_{ij}=f_{ij}(z,z_{1},...,z_{k})$
satisfying \textbf{(i-ii)} with $\left(f_{31,z_{2}}\right)^{2}-\delta\left(f_{21,z_{2}}\right)^{2}=0$,
if and only if it can be written in one of the following two forms:\\

\noindent \textbf{(I) } 
\[
z_{2t}=D_{x}\psi+a\,\psi+b\,z_{2}\pm n\,(1-a)\alpha
\]

\noindent with $\delta=1$ and associated 1-forms 
\[
\begin{array}{l}
\omega_{1}=\mp a\,dx+\left[\left(1-a\right)\left(mz_{1}+n\right)\pm b\right]\,dt\,,\vspace{5pt}\\
\omega_{2}=\left(z_{2}+\alpha\right)\,dx+\left[\psi\mp\left(1-a\right)\left({\displaystyle \frac{m}{2}z_{1}^{2}}+n\,z_{1}\right)-{\displaystyle \alpha\,b}\right]\,dt\,,\vspace{5pt}\\
\omega_{3}=\pm\,\omega_{2}-(1-a)\,m\,dt,
\end{array}
\]
where $\psi=\psi(z,z_{1},z_{2})$ is a real differentiable function,
whereas $a,b\in\{0,1\}$ and $\alpha,m,n\in\mathbb{R}$ are such that
$(a-1)\alpha m=0$, $(a-1)b=0$ and, in addition, $a\psi+bz_{2}+m^{2}+n^{2}\neq0$;

\vspace*{0.4cm}

\noindent \textbf{(II) } 
\[
z_{2t}=\mp D_{x}^{3}h-D_{x}\left(\left(z_{2}+m\right)\,h\right)-z_{2}\,D_{x}\left(h\right)
\]

\noindent with $\delta=1$ and associated 1-forms{\small{}{}{}{}{}{}{}
\[
\begin{array}{l}
\omega_{1}=\eta\,dx+\left(\mp D_{x}h-\eta\,h\right)\,dt\,,\vspace{5pt}\\
\omega_{2}=\left({\displaystyle \frac{z_{2}}{\alpha\mp\beta}}+\beta\right)\,dx+\left({\displaystyle \frac{\mp D_{x}^{2}h-\eta\,D_{x}h}{\alpha\mp\beta}}-\left({\displaystyle \frac{z_{2}}{\alpha\mp\beta}}+\beta\right)h\right)\,dt\,,\vspace{5pt}\\
\omega_{3}=\left(\pm{\displaystyle \frac{z_{2}}{\alpha\mp\beta}}+\alpha\right)\,dx+\left({\displaystyle \frac{-D_{x}^{2}h\mp\eta\,D_{x}h}{\alpha\mp\beta}}\mp\left({\displaystyle \frac{z_{2}}{\alpha\mp\beta}}\pm\alpha\right)h\right)\,dt,
\end{array}
\]
} where $\alpha,\beta,\eta\in\mathbb{R}$, $\alpha\mp\beta\neq0$,
$m:=\pm\left(\alpha^{2}-\beta^{2}-\eta^{2}\right)$ and $h=h(z)$
is a real differentiable function satisfying $h'\neq0$. 
\end{thm}

\medskip{}

\begin{thm}
\label{teo.clas.3.3*} A partial differential equation of the form
\eqref{eq:C-2} describes \textbf{pss} or \textbf{ss} ($\delta=1$
or $\delta=-1,$ respectively), with corresponding functions $f_{ij}=f_{ij}(z,z_{1},...,z_{k})$
satisfying \textbf{(i-ii)} with $\left(f_{31,z_{2}}\right)^{2}-\delta\left(f_{21,z_{2}}\right)^{2}\neq0$
and $\eta=0$, if and only if it can be written in one of the following
two forms:\\

\noindent \textbf{(I)} 
\[
z_{2t}=D_{x}\psi
\]
with $\delta=\pm1$ and associated 1-forms{\small{}{}{}{}{}{}{}
\[
\begin{array}{l}
\omega_{1}={\displaystyle -}\phi'\,dt\,,\vspace{8pt}\\
\omega_{2}={\displaystyle z_{2}\,dx+\left(\psi\pm r\,\phi\right)\,dt\,},\vspace{8pt}\\
\omega_{3}={\displaystyle \pm r\,z_{2}\,dx+\left(\pm r\,\psi+\delta\,\phi\right)\,dt}
\end{array}
\]
} where $r\in\mathbb{R}$ and $\psi=\psi(z,z_{1},z_{2})$, $\phi=\phi(z_{1})$
are real differentiable functions such that $\psi$ is not constant
and $\phi''=-\left(r^{2}-\delta\right)\,\phi$, $\phi\neq0$, i.e.,
\[
\phi=\left\{ \begin{array}{lc}
a\,\cos\left(\sqrt{r^{2}-\delta}\,z_{1}\right)+b\,\sin\left(\sqrt{r^{2}-\delta}\,z_{1}\right)\vspace{7pt}\qquad & {\displaystyle r^{2}-\delta}>0\\
a\,\cosh\left(\sqrt{-r^{2}+\delta}\,z_{1}\right)+b\,\sinh\left(\sqrt{-r^{2}+\delta}\,z_{1}\right)\qquad & r^{2}-\delta<0
\end{array}\right.
\]
with $a.b\in\mathbb{R}$, $a^{2}+b^{2}\neq0$;\vspace*{0.4cm}

\noindent \textbf{(II)} 
\[
z_{2t}=-D_{x}^{2}\left(2\,\gamma\,z_{1}\varphi'\right)-D_{x}\left(\left(z_{2}+m\right)\varphi\right)+2\,\delta\,r^{2}\,z_{1}\varphi'
\]
with $\delta=\pm1$ and associated 1-forms{\small{}{}{}{}{}{}{}
\[
\begin{array}{l}
\omega_{1}={\displaystyle -2\,r}\,z_{1}\varphi'\,dt\,,\vspace{8pt}\\
\omega_{2}=\left({\displaystyle \frac{\mu}{\gamma}}\,\left(z_{2}+m\right)\mp{\displaystyle r\,\frac{\sqrt{\gamma+\delta\mu^{2}}}{\gamma}}\right)\,dx+\left[-\mu D_{x}\left({\displaystyle 2z_{1}\varphi'}\right)+\left(-{\displaystyle \frac{\mu}{\gamma}}\,\left(z_{2}+m\right)\pm{\displaystyle r\,\frac{\sqrt{\gamma+\delta\mu^{2}}}{\gamma}}\right)\,\varphi\right]\,dt\,,\vspace{8pt}\\
\omega_{3}=\left({\displaystyle \pm\frac{\sqrt{\gamma+\delta\mu^{2}}}{\gamma}}\,\left(z_{2}+m\right)-{\displaystyle \delta\,r\,\frac{\mu}{\gamma}}\right)\,dx+\vspace{5pt}\\
\qquad\qquad+\left[\mp\sqrt{\gamma+\delta\mu^{2}}D_{x}\left({\displaystyle 2z_{1}\varphi'}\right)+\left(\mp{\displaystyle \frac{\sqrt{\gamma+\delta\mu^{2}}}{\gamma}}\,\left(z_{2}+m\right)+{\displaystyle {\displaystyle \delta\,r\,\frac{\mu}{\gamma}}}\right)\,\varphi\right]\,dt\,,
\end{array}
\]
} where $m,r,\gamma,\mu\in\mathbb{R}$ and $\varphi=\varphi(2\,m\,z+z_{1}^{2})$
is a differentiable function such that $\gamma\neq0,$ $\gamma+\delta\mu^{2}\geq0$
and $\varphi'\neq0$. 
\end{thm}

\medskip{}

\begin{thm}
\label{teo.clas.3.5*} A partial differential equation of the form
\eqref{eq:C-2} describes \textbf{pss} or \textbf{ss} ($\delta=1$
or $\delta=-1,$ respectively), with corresponding functions $f_{ij}=f_{ij}(z,z_{1},...,z_{k})$
satisfying \textbf{(i-ii)} with $\left(f_{31,z_{2}}\right)^{2}-\delta\left(f_{21,z_{2}}\right)^{2}\neq0$
and $\eta\neq0$, if and only if it can be written in one of the following
two forms: \\

\textbf{(I)} 
\[
z_{2t}=D_{x}^{2}\left(\varphi-\frac{r}{\eta^{2}}\,e^{-z_{1}}\right)-D_{x}\left(z_{2}\,\varphi\right)+r\,e^{-z_{1}}
\]
\begin{flushleft}
with $\delta=1$ and associated 1-forms 
\[
\begin{array}{l}
\omega_{1}=\eta\,\left(dx-\varphi\,dt\right)\,,\vspace{5pt}\\
\\
\omega_{2}=\left(-\rho\,z_{2}\pm\eta\,\sqrt{\rho^{2}-1}\right)\,dx\\
\hspace*{3cm}+\left[-\rho\,D_{x}(\varphi)+\left(\rho\,z_{2}\mp\eta\,\sqrt{\rho^{2}-1}\right)\,\left(\varphi-{\displaystyle \frac{r}{\eta^{2}}}e^{-z_{1}}\right)\right]\,dt\,,\vspace{5pt}\\
\\
\omega_{3}=\left(\mp\sqrt{\rho^{2}-1}\,z_{2}+\eta\,\rho\right)\,dx\\
\hspace*{3cm}+\left[\mp\sqrt{\rho^{2}-1}\,D_{x}(\varphi)+\left(\pm\sqrt{\rho^{2}-1}\,z_{2}-\eta\,\rho\right)\,\left(\varphi-{\displaystyle \frac{r}{\eta^{2}}}e^{-z_{1}}\right)\right]\,dt,
\end{array}
\]
\par\end{flushleft}
\begin{flushleft}
\vspace*{0.3cm}
 where $\eta,\rho,r\in\mathbb{R}$, $\eta\neq0$, $\rho^{2}-1\geq0$
and $\varphi=\varphi(z,z_{1})$ is a differentiable function such
that $D_{x}\varphi\neq0$; 
\par\end{flushleft}
\noindent \begin{flushleft}
\vspace*{0.5cm}
 \textbf{(II)} 
\par\end{flushleft}
\noindent \begin{flushleft}
\[
\begin{array}{ll}
z_{2t}= & \frac{1}{\eta}D_{x}^{2}\left({\displaystyle \frac{2\eta\gamma z_{1}\ell'-r\,\ell}{\gamma\eta^{2}+r^{2}}}\right)+\frac{1}{\eta}D_{x}\left({\displaystyle \frac{\left(2r\,z_{1}\ell'+\eta\ell\right)}{\gamma\eta^{2}+r^{2}}\,\left(z_{2}+m\right)}\right)-2\delta z_{1}\ell'\end{array}
\]
\par\end{flushleft}
\begin{flushleft}
with $\delta=\pm1$ and associated 1-forms {\small{}{}{}{}{}{}{}
\[
\begin{array}{l}
\omega_{1}=\eta\,dx+\left({\displaystyle \frac{2\,r\,z_{1}\ell'+\eta\ell}{\gamma\eta^{2}+r^{2}}}\right)\,dt\,,\vspace{8pt}\\
\omega_{2}=\left({\displaystyle \frac{{\displaystyle \mu}}{\gamma}}\,(z_{2}+m)\mp\,r\,{\displaystyle \frac{\sqrt{\gamma+\delta\mu^{2}}}{\gamma}}\right)\,dx\vspace{3pt}\\
\qquad+\left({\displaystyle \frac{{\displaystyle \mu}}{\eta\,\gamma}}D_{x}\left({\displaystyle \frac{2\eta\gamma z_{1}\ell'-r\,\ell}{\gamma\eta^{2}+r^{2}}}\right)+{\displaystyle \frac{\mu}{\eta\,\gamma}}\,{\displaystyle \frac{\left(2r\,z_{1}\ell'+\eta\ell\right)}{\gamma\eta^{2}+r^{2}}}\left(z_{2}+m\right)\pm{\displaystyle {\displaystyle \frac{\sqrt{\gamma+\delta\mu^{2}}}{\gamma}}}\,\left({\displaystyle \frac{2\eta\gamma z_{1}\ell'-r\,\ell}{\gamma\eta^{2}+r^{2}}}\right)\right)\,dt\,,\vspace{8pt}\\
\omega_{3}=\pm\left({\displaystyle {\displaystyle \frac{\sqrt{\gamma+\delta\mu^{2}}}{\gamma}}}\,(z_{2}+m)\mp{\displaystyle \frac{\delta\,\mu\,r}{\gamma}}\right)\,dx\vspace{3pt}\\
\qquad\pm\left({\displaystyle \frac{\sqrt{\gamma+\delta\mu^{2}}}{\eta\,\gamma}}\,D_{x}\left({\displaystyle {\displaystyle \frac{2\eta\gamma z_{1}\ell'-r\,\ell}{\gamma\eta^{2}+r^{2}}}}\right)+{\displaystyle \frac{\sqrt{\gamma+\delta\mu^{2}}}{\eta\,\gamma}}\,{\displaystyle \frac{\left(2r\,z_{1}\ell'+\eta\ell\right)}{\gamma\eta^{2}+r^{2}}\,\left(z_{2}+m\right)}\pm{\displaystyle \frac{\delta\,\mu}{\gamma}}\,\left({\displaystyle \frac{2\eta\gamma z_{1}\ell'-r\,\ell}{\gamma\eta^{2}+r^{2}}}\right)\right)\,dt
\end{array}
\]
}{\small\par}
\par\end{flushleft}
\begin{flushleft}
 
\par\end{flushleft}
\begin{flushleft}
\vspace*{0.3cm}
 where $r,\gamma,\mu,\eta\in\mathbb{R}$ satisfy $\gamma+\delta\mu^{2}\geq0$
and $\left(\gamma\eta^{2}+r^{2}\right)\gamma\,\eta\neq0$, whereas
$\ell=\ell(z_{1}^{2}+2\,m\,z)$ is a differentiable function such
that $\ell'\neq0$. 
\par\end{flushleft}

\end{thm}

\vspace*{0.3cm}

\subsection{Examples\label{subsec:examples}}

\medskip{}

Here we collect some examples of equations described by Theorems \ref{teo.clas.3.2}-\ref{teo.clas.3.5*},
where the corresponding ZCRs (or linear problems) admit some free
parameter. 
\begin{example}
\label{ex3.3} In Theorem \ref{teo.clas.3.2} - (II), by choosing
$h(z)=z+m$, $m\in\mathbb{R}$, one gets the nonlinear partial differential
equation 
\begin{eqnarray}
z_{t}-\lambda\,z_{2t}=(\lambda\,z+\lambda\,m\mp1)\,z_{3}+2\,\lambda\,z_{1}\,z_{2}-3\,z\,z_{1}-m\,z_{1},\quad\lambda\in\mathbb{R},\label{CH.eq.}
\end{eqnarray}

\noindent which describes \textbf{pss} with associated 1-forms{\small{}{}{}{}{}{}
\[
\begin{array}{l}
\omega_{1}=\eta\,dx+\left(\mp\,z_{1}-\eta\,\left(z+m\right)\right)\,dt\,,\vspace{5pt}\\
\omega_{2}={\displaystyle \frac{1}{\alpha}}\left[\left(z-\lambda z_{2}\pm{\displaystyle \frac{\eta^{2}}{2}}\mp{\displaystyle \frac{\alpha^{2}}{2}}\right)\,dx+\left(\mp z_{2}-\eta\,z_{1}-\left(z-\lambda z_{2}\pm{\displaystyle \frac{\eta^{2}}{2}}\mp{\displaystyle \frac{\alpha^{2}}{2}}\right)\left(z+m\right)\right)\,dt\right]\,,\vspace{5pt}\\
\omega_{3}=\pm{\displaystyle \frac{1}{\alpha}}\left[\left(z-\lambda z_{2}\pm{\displaystyle \frac{\eta^{2}}{2}}\pm{\displaystyle \frac{\alpha^{2}}{2}}\right)\,dx+\left(\mp z_{2}-\eta\,z_{1}-\left(z-\lambda z_{2}\pm{\displaystyle \frac{\eta^{2}}{2}}\pm{\displaystyle \frac{\alpha^{2}}{2}}\right)\left(z+m\right)\right)\,dt\right],
\end{array}
\]
}{\small\par}

\noindent where $\alpha\in\mathbb{R}-\{0\}$ and $\eta\in\mathbb{R}$. 

One can see that, up to a translation of
$z$ by a constant, (\ref{CH.eq.}) is the CH-r equation described
in \citep{DGH, RaSc}, when it depends on $z_{t}$.

\noindent In particular, for $\lambda=1$ and $m=\pm1$ one gets from
\eqref{CH.eq.} the Camassa-Holm equation \citep{DGH} 
\[
z_{t}-z_{2t}=z\,z_{3}+2\,z_{1}\,z_{2}-3\,z\,z_{1}\mp\,z_{1}.
\]
Also, by taking $\lambda=m=0$ in \eqref{CH.eq.} one gets the \textit{KdV
equation} in the following form 
\[
z_{t}=z_{3}-3\,z\,z_{1}.
\]

In these examples one can check that $\alpha$ is a free parameter,
i.e., a parameter appearing in the $1$-forms but not in the equation,
that cannot be removed by gauge transformations. 
\end{example}

\noindent \medskip{}

\begin{example}
In Theorem \ref{teo.clas.3.2*} -- (I), with $a=\alpha=0$, up to
a coordinate transformation $x\mapsto\eta\,x$, $\eta$ $\in$ $\mathbb{R}-\{0\}$,
and further rearranging $\psi$ according to $\psi\mapsto\psi/\eta$,
one gets the nonlinear partial differential equation 
\[
z_{2t}=D_{x}\psi
\]
which describes \textbf{pss} with associated 1-forms 
\[
\begin{array}{l}
\omega_{1}=\left(m\,\eta\,z_{1}+n\right)\,dt\,,\vspace{5pt}\\
\omega_{2}=\eta\,z_{2}\,dx+\eta\,\left(\psi\mp{\displaystyle \frac{m\,\eta}{2}z_{1}^{2}}\mp\,n\,z_{1}\right)\,dt\,,\vspace{5pt}\\
\omega_{3}=\pm\,\omega_{2}-\,m\,dt,
\end{array}
\]
where $\psi=\psi(z,z_{1},z_{2})$ is a real differentiable function
and $m,n\in\mathbb{R}$ are such that $m^{2}+n^{2}\neq0$. In
this example, one can check that $\eta$ is a free parameter that
cannot be removed by gauge transformations. 
\end{example}

\medskip{}

\begin{example}
In Theorem \ref{teo.clas.3.2*} -- (II), by choosing $\alpha=0$
and $\beta=\pm\sqrt{r-\eta^{2}}$, with $r,\eta\in\mathbb{R}$ such
that $r-\eta^{2}>0$, one gets the nonlinear partial differential
equation 
\begin{equation}
z_{2t}=\mp D_{x}^{3}h-D_{x}\left(\left(z_{2}\mp r\right)\,h\right)-z_{2}\,D_{x}\left(h\right)\label{eq_altCHR}
\end{equation}
which describes \textbf{pss} with associated 1-forms{\small{}{}{}{}{}{}{}
\begin{equation}
\begin{array}{l}
\omega_{1}=\eta\,dx+\left(\mp D_{x}h-\eta\,h\right)\,dt\,,\vspace{5pt}\\
\omega_{2}={\displaystyle \frac{-z_{2}\pm\left(r-\eta^{2}\right)}{\sqrt{r-\eta^{2}}}}\,dx+{\displaystyle \frac{\pm D_{x}^{2}h+\eta\,D_{x}h+\left(z_{2}\,\pm\left(r-\eta^{2}\right)\right)\,h}{\sqrt{r-\eta^{2}}}}\,dt\,,\vspace{5pt}\\
\omega_{3}=\mp{\displaystyle \frac{z_{2}}{\sqrt{r-\eta^{2}}}}\,dx+{\displaystyle \frac{D_{x}^{2}h\pm\eta\,D_{x}h\pm z_{2}\,h}{\sqrt{r-\eta^{2}}}}\,dt,
\end{array}\label{eq_omega_alt_CHr}
\end{equation}
with $h=h(z)$} a differentiable function such that $h'\neq0$. 

In particular, if $h=z+m$ equation (\ref{eq_altCHR})
reduces to an equation that, up to a translation of $z$ by a constant,
has the form of CH-r equation \citep{DGH, RaSc} with no term in $z_{t}$.
Also, one can check that $\eta$ in
(\ref{eq_omega_alt_CHr}) is a free parameter that cannot be removed
by gauge transformations. 
\end{example}

\medskip{}

\begin{example}
In Theorem \ref{teo.clas.3.3*} -- (I), by choosing $\psi(z,z_{1},z_{2})=\alpha\,D_{x}(z\,z_{1})+\beta\,z$,
$\alpha,\beta\in\mathbb{R}$, $\alpha^{2}+\beta^{2}\neq0$, one gets
the nonlinear partial differential equation 
\begin{equation}
z_{2t}=\alpha\,z\,z_{3}+3\,\alpha\,z_{1}\,z_{2}+\beta\,z_{1}\label{eq_Kraenkel}
\end{equation}
which describes \textbf{pss} or \textbf{ss} with associated 1-forms{\small{}{}{}{}{}{}{}
\[
\begin{array}{l}
\omega_{1}={\displaystyle -}\phi'\,dt\,,\vspace{8pt}\\
\omega_{2}={\displaystyle z_{2}\,dx+\left(\alpha\,z\,z_{2}+\alpha\,z_{1}^{2}+\beta\,z\pm r\,\phi\right)\,dt\,},\vspace{8pt}\\
\omega_{3}={\displaystyle \pm r\,z_{2}\,dx+\left[\pm r\,(\alpha\,z\,z_{2}+\alpha\,z_{1}^{2}+\beta\,z)+\delta\,\phi\right]\,dt}
\end{array}
\]
}where $r\in\mathbb{R}$ and $\phi=\phi(z_{1})$ is a real differentiable
function satisfying $\phi''=-\left(r^{2}-\delta\right)\,\phi$, $\phi\neq0$,
i.e., 
\begin{equation}
\phi=\left\{ \begin{array}{lc}
a\,\cos\left(\sqrt{r^{2}-\delta}\,z_{1}\right)+b\,\sin\left(\sqrt{r^{2}-\delta}\,z_{1}\right),\vspace{7pt}\qquad & {\displaystyle r^{2}-\delta}>0,\\
a\,\cosh\left(\sqrt{-r^{2}+\delta}\,z_{1}\right)+b\,\sinh\left(\sqrt{-r^{2}+\delta}\,z_{1}\right),\qquad & r^{2}-\delta<0,
\end{array}\right.\label{phi_th3.3}
\end{equation}
with $a,b\in\mathbb{R}$, $a^{2}+b^{2}\neq0$.

We note here that, a particular instance of (\ref{eq_Kraenkel}) is
the equation 
\[
2\sqrt{\frac{k}{g}}z_{2t}=k^{2}z_{x}-\frac{3}{2}k\left(zz_{1}\right)_{xx}\,,
\]
where $k,g\in\mathbb{R}-\{0\}$, $k/g>0$, that describes surface
waves in deep water \citep{K.et.al}. In this example,
one can check that $r$ is a free parameter that cannot be removed
by gauge transformations. 
\end{example}

\medskip{}

\begin{example}
Likewise equation (\ref{eq_Kraenkel}), any equation of the form (I)
in Theorem \ref{teo.clas.3.3*}, can be rewritten as 
\begin{eqnarray}
D_{x}\left(z_{1t}-\psi(z,z_{1},z_{2})\right)=0,\label{eq_th3.3*}
\end{eqnarray}
and hence belongs to the system of differential consequences of 
\begin{eqnarray}
z_{1t}-\psi(z,z_{1},z_{2})=0.\label{general.eq}
\end{eqnarray}
It is noteworthy to remark here that the $1$-forms 
\begin{equation}
\begin{array}{l}
\omega_{1}={\displaystyle -}\phi'\,dt\,,\vspace{8pt}\\
\omega_{2}={\displaystyle z_{2}\,dx+\left(\psi\pm r\,\phi\right)\,dt\,},\vspace{8pt}\\
\omega_{3}={\displaystyle \pm r\,z_{2}\,dx+\left(\pm r\,\psi+\delta\,\phi\right)\,dt},
\end{array}\label{forme_th3.3}
\end{equation}
provided by Theorem \ref{teo.clas.3.3*} -- (I), with $\phi$ given
by (\ref{phi_th3.3}), satisfy the structure equations (\ref{struttura})
if and only if $z$ is a solution of (\ref{eq_th3.3*}). Hence it
is not true that equation (\ref{general.eq}) describes \textbf{pss}
or \textbf{ss} with associated $1$-forms (\ref{forme_th3.3}). However,
$\Omega=Xdx+Tdt$ satisfy the zero-curvature condition modulo a differential
consequence of (\ref{general.eq}), hence the $1$-forms (\ref{forme_th3.3})
determine a ZCR (or linear problem) for (\ref{general.eq}).

Here are other instances of \eqref{general.eq} that admit a ZCR $\Omega=Xdx+Tdt$
provided by (\ref{forme_th3.3}): 
\begin{enumerate}
\item[(i)] if $\psi(z,z_{1},z_{2})=-z\,z_{2}-\Phi(z_{1})$, then one gets the
Calogero \citep{Cal} equation 
\begin{equation}
z_{1t}+z\,z_{2}+\Phi(z_{1})=0.\label{Calogero.eq}
\end{equation}
\item[(ii)] if $\psi(z,z_{1},z_{2})=-z\,z_{2}-\frac{1}{2}z_{1}^{2}$ (or $\Phi(z_{1})=\frac{1}{2}z_{1}^{2}$
in Calogero equation), then one gets the Hunter-Saxton \citep{RaSc}
equation 
\[
z_{1t}+z\,z_{2}+\frac{1}{2}z_{1}^{2}=0.
\]
\item[(iii)] if $\psi(z,z_{1},z_{2})=F(z)$, then one gets the following celebrated
equations 
\begin{eqnarray*}
\begin{array}{llc}
 & z_{1t}-e^{z}+e^{-2\,z}=0, & \qquad\textnormal{Tzitz\' eica \citep{TI}}\vspace*{0.3cm}\\
 & z_{1t}+e^{-z}-e^{2\,z}=0, & \qquad\textnormal{Bullough--Dodd--Zhiber--Shabat \citep{Z-S.eq}}\vspace*{0.3cm}\\
 & z_{1t}+e^{z}+e^{-2\,z}=0, & \qquad\textnormal{Dodd--Bullough--Mikhailov \citep{Waz}}\vspace*{0.3cm}\\
 & z_{1t}-e^{-z}-e^{-2\,z}=0. & \qquad\textnormal{Tzitz\' eica--Dodd--Bullough \citep{Waz}}
\end{array}
\end{eqnarray*}
We remark here that Dodd and Bullough \citep{DB} referred to $z_{1t}=F(z)$
as a generalized sine--Gordon equation, which includes the classical
sine--Gordon $z_{1t}=\sin(z)$, the multiple type sine--Gordon $z_{1t}=\sin(z)+\frac{1}{2}\,\sin(\frac{z}{2})$
and the Liouville $z_{1t}=e^{z}$ equations. 
\item[(iv)] if $\psi(z,z_{1},z_{2})=z+\dfrac{1}{6}\,(z^{3})_{,xx}$, then one
gets the Rabelo \citep{R} equation 
\[
z_{1t}-z-\dfrac{1}{6}\,(z^{3})_{,xx}=0,
\]
which appeared in \citep{SW} describing the propagation of ultra-short
light pulses in silica optical fibers (see algo \citep{Sak}). 
\end{enumerate}
\end{example}

\medskip{}

\begin{example}
It is quite interesting that in the cases (I) of Theorems \ref{teo.clas.3.5}
and \ref{teo.clas.3.5*} a wise inspection leads to a system of $1$-forms
$\omega_{1},\omega_{2},\omega_{3}$ that unify the two cases and describe
equations of the form $\nu\,z_{t}-\lambda\,z_{2t}=A\,z_{3}+B$ with
$\nu$ ranging in the whole $\mathbb{R}$.

Indeed, by introducing the auxiliary variable $w$ such that $w_{1}=z$
(equivalently, $w=\int z\,dx$), one obtains the nonlinear partial
differential equation 
\begin{equation}
\nu\,z_{t}-\lambda\,z_{2t}=D_{x}^{2}\left(\varphi-\frac{r}{\eta^{2}}\,e^{\lambda z_{1}-\nu\,w}\right)-D_{x}\left(\left(\nu\,z-\lambda z_{2}\right)\,\varphi\right)+r\,e^{\lambda z_{1}-\nu\,w}\label{eq_ex_nonlocal}
\end{equation}
which describes \textbf{pss} with associated 1-forms 
\begin{equation}
\begin{array}{l}
\omega_{1}=\eta\,\left(dx-\varphi\,dt\right)\,,\vspace{5pt}\\
\\
\omega_{2}=\left({\displaystyle -\rho}\,(\nu\,z-\lambda\,z_{2})\pm\eta\,\sqrt{\rho^{2}-1}\right)\,dx\\
\hspace*{3cm}+\left[{\displaystyle -\rho}D_{x}\left(\varphi\right)+\left(\rho\,\left(\nu\,z-\lambda z_{2}\right)\mp\eta\,\sqrt{\rho^{2}-1}\right)\,\left(\varphi-{\displaystyle \frac{r}{\eta^{2}}}e^{\lambda z_{1}-\nu\,w}\right)\right]\,dt\,,\vspace{5pt}\\
\\
\omega_{3}=\left({\displaystyle \mp\sqrt{\rho^{2}-1}}\,(\nu\,z-\lambda\,z_{2})+\rho\,\eta\right)\,dx\\
\hspace*{3cm}+\left[{\displaystyle \mp\sqrt{\rho^{2}-1}}\,D_{x}\left(\varphi\right)+\left(\pm{\displaystyle \sqrt{\rho^{2}-1}}\,\left(\nu\,z-\lambda z_{2}\right)-\eta\,\rho\right)\,\left(\varphi-{\displaystyle \frac{r}{\eta^{2}}}e^{\lambda z_{1}-\nu\,w}\right)\right]\,dt\,,
\end{array}\label{eq_ex_nonlocal-1}
\end{equation}
where $\eta,\rho,r\in\mathbb{R}$, $\eta\neq0$, $\rho^{2}-1\geq0$
and $\varphi=\varphi(z,w,z_{1})$ is a differentiable function such
that $D_{x}\varphi\neq0$.

Thus Theorems \ref{teo.clas.3.5} and \ref{teo.clas.3.5*} describe
particular cases of a most general class, whose description requires
the introduction of $w$, when $\nu\neq0$. Hence, since $w=\int z\,dx$,
when $\nu\neq0$ equation (\ref{eq_ex_nonlocal}) is an integrodifferential
equation.

An interesting instance of this class results from the choice 
\[
\rho=\frac{r}{2\,\eta^{2}},\qquad\varphi=G+\frac{r\,e^{\lambda z_{1}-\nu\,w}+e^{-\lambda z_{1}+\nu\,w}}{2\,\eta^{2}}
\]
where $r\in\mathbb{R}-\{0\}$ and $G=G(z,w,z_{1})$ is a differentiable
function, that does not depend on $\eta$. Indeed, in this case one
gets the equation 
\[
\nu\,z_{t}-\lambda\,z_{2t}=D_{x}^{2}G-D_{x}\left(\left(\nu\,z-\lambda z_{2}\right)\,G\right)+r\,e^{\lambda z_{1}-\nu\,w}
\]
with the associated $1$-forms 
\[
\begin{array}{l}
\omega_{1}=\eta\,dx-\left({\displaystyle \frac{2\,\eta^{2}G+r{\rm e}^{\lambda\,z_{1}-\nu\,w}+{\rm e}^{-\lambda\,z_{1}+\nu\,w}}{2\eta}}\right)\,dt\,,\vspace{5pt}\\
\\
\omega_{2}=\left({\displaystyle -\frac{r}{2\eta^{2}}}\,(\nu\,z-\lambda\,z_{2})\pm{\displaystyle \frac{1}{2\eta}}\,\sqrt{r^{2}-4\eta^{4}}\right)\,dx\vspace{3pt}\\
\hspace*{3cm}+\left[{\displaystyle \frac{\mp\,\left(2\eta^{2}G-r{\rm e}^{\lambda\,z_{1}-\nu\,w}+{\rm e}^{-\lambda\,z_{1}+\nu\,w}\right)\sqrt{r^{2}-4\eta^{4}}+2\,\eta\,r\,\left((\nu\,z-\lambda\,z_{2})G-D_{x}G\right)}{4\eta^{3}}}\right]\,dt\,,\vspace{5pt}\\
\\
\omega_{3}=\left(\mp{\displaystyle \frac{1}{2\eta^{2}}}\,(\nu\,z-\lambda\,z_{2})\sqrt{r^{2}-4\eta^{4}}{\displaystyle +\frac{r}{2\eta}}\right)\,dx\vspace{3pt}\\
\hspace*{3cm}+\left[{\displaystyle \frac{\pm\,2\,\eta\,\left((\nu\,z-\lambda\,z_{2})G-D_{x}G\right)\sqrt{r^{2}-4\eta^{4}}-2\eta^{2}\,r\,G+r^{2}{\rm e}^{\lambda\,z_{1}-\nu\,w}-r{\rm e}^{-\lambda\,z_{1}+\nu\,w}}{4\eta^{3}}}\right]\,dt\,,
\end{array}
\]

where $\eta$ is a free parameter that cannot be removed by gauge
transformations.

As an example, by taking $G=-h\,z_{1}$, $h\in\mathbb{R}$, one gets
the equation 
\[
\nu\,z_{t}-\lambda\,z_{2t}=-h\,\left(\lambda\left(z_{1}\,z_{3}+z_{2}^{2}\right)-\nu\,z\,z_{2}-\nu\,z_{1}^{2}+z_{3}\right)+r\,e^{\lambda z_{1}-\nu\,w}.
\]
Notice that for $h=0,\nu=0,\lambda=-1,r=1$ this equation reduces
to $z_{2t}=e^{z_{1}}$, which is linked to the Liouville equation
$\tilde{z}_{12}=e^{\tilde{z}}$ by differential substitution $\tilde{z}=z_{1}$. 
\end{example}

\begin{rem}
According to the literature (see for instance \citep{KrV}) variables
like $w$, used in (\ref{eq_ex_nonlocal}-\ref{eq_ex_nonlocal-1})
and defined by means of a first order differential equation for $w$,
can be referred to as \textit{non-local variables}. 
\end{rem}

\medskip{}
 {\small{}{}{}{}{}{}{}}{\small\par}

\section{Proof of the main results for equations of the form $z_{t}-\lambda\,z_{2t}=A\,z_{3}+B$
\label{sec4}}

\noindent This section is devoted to the proofs of Theorems \ref{teo.clas.3.2},
\ref{teo.clas.3.3} and \ref{teo.clas.3.5}.

\subsection{Preliminary lemma}

We start with a lemma that gives a convenient characterization of
equations (\ref{eq:C-1}) describing \textbf{pss} or \textbf{ss} ($\delta=1$
or $\delta=-1,$ respectively) under assumptions \textbf{(i-ii)}.

\medskip{}

\begin{lem}
\label{lema1} A partial differential equation of the form \eqref{eq:C-1}
describes \textbf{pss} or \textbf{ss} ($\delta=1$ or $\delta=-1,$
respectively), with corresponding functions $f_{ij}=f_{ij}(z,z_{1},...,z_{k})$
satisfying \textbf{(i-ii)}, if and only if 
\begin{equation}
f_{11}=\eta,\quad\quad f_{p1}=\mu_{p}\,z_{2}+\sigma_{p}\,z+\alpha_{p},\quad2\leq p\leq3,\label{fij}
\end{equation}

\begin{equation}
\begin{array}{lcccc}
f_{12}=f_{12}(z,z_{1}),\qquad & \mu_{p}=-\lambda\sigma_{p},\quad & f_{p2}=f_{p2}(z,z_{1},z_{2}),\quad & p=2,3,\end{array}\label{7.2-2.3.1-cond_A}
\end{equation}
\begin{eqnarray}
(\sigma_{2}^{2}+\sigma_{3}^{2})\,A=\sigma_{2}\,f_{22,z_{2}}+\sigma_{3}\,f_{32,z_{2}}\,,\label{A001}
\end{eqnarray}
\begin{eqnarray}
\sigma_{3}f_{22}-\sigma_{2}f_{32} & = & \varphi,\label{first.f22.f32}\\
\alpha_{3}\,f_{22}-\alpha_{2}\,f_{32} & = & {\displaystyle \sum_{i=0}^{1}z_{i+1}f_{12,z_{i}}-(z-\lambda z_{2})\,\varphi,\label{sec.f22.f32}}
\end{eqnarray}
\begin{eqnarray}
(\sigma_{2}^{2}+\sigma_{3}^{2})\,B={\displaystyle \sum_{i=0}^{1}z_{i+1}(\sigma_{2}\,f_{22}+\sigma_{3}\,f_{32})_{,z_{i}}-\eta(\sigma_{2}\,f_{32}+\delta\sigma_{3}\,f_{22})\qquad\qquad\qquad\nonumber}\\
+[\sigma_{2}\sigma_{3}(1+\delta)(z-\lambda z_{2})+\sigma_{2}\alpha_{3}+\delta\sigma_{3}\alpha_{2}]\,f_{12},\label{B001}\\
\eta(\sigma_{3}\,f_{32}-\delta\sigma_{2}\,f_{22})={\displaystyle \sum_{i=0}^{1}z_{i+1}\,\varphi_{,z_{i}}+[(\sigma_{3}^{2}-\delta\sigma_{2}^{2})(z-\lambda z_{2})+\sigma_{3}\alpha_{3}-\delta\sigma_{2}\alpha_{2}]\,f_{12},\label{last.eq.f22.f32}}
\end{eqnarray}
where $\eta,\mu_{p},\sigma_{p},\alpha_{p}\in\mathbb{R}$, $\sigma_{2}^{2}+\sigma_{3}^{2}\neq0$,
and $\varphi=\varphi(z,z_{1})$ is a real and differentiable function
of $z$ and $z_{1}$. Furthermore, 
\begin{eqnarray}
\eta f_{22}-[\sigma_{2}(z-\lambda z_{2})+\alpha_{2}]\,f_{12}\neq0.\label{7.5.1}
\end{eqnarray}
\end{lem}

\begin{proof}
Let $z=z(x,t)$ be a function satisfying ($\ref{eq:C-1}$). Then
\begin{eqnarray}
 &  & dz\wedge dx=-(A\,z_{3}+B)dx\wedge dt+\lambda\,dz_{2}\wedge dx,\label{(7.6)}\\
 &  & dz_{i}\wedge dt=z_{i+1}dx\wedge dt,\quad0\leq i\leq m.\nonumber 
\end{eqnarray}
Since the 1-forms $\omega_{i}$ satisfy the structure equations ($\ref{struttura}$),
one gets 
\begin{eqnarray}
df_{11}\wedge dx+df_{12}\wedge dt+(f_{32}f_{21}-f_{31}f_{22})dx\wedge dt=0,\nonumber \\
df_{21}\wedge dx+df_{22}\wedge dt+(-f_{11}f_{32}+f_{31}f_{12})dx\wedge dt=0,\label{(7.7-a)}\\
df_{31}\wedge dx+df_{32}\wedge dt+\delta(f_{21}f_{12}-f_{11}f_{22})dx\wedge dt=0.\nonumber 
\end{eqnarray}
Then, by substituting 
\[
df_{ij}={\displaystyle \sum_{k=0}^{m}f_{ij,z_{k}}dz_{k},}
\]
in ($\ref{(7.7-a)}$), using ($\ref{(7.6)}$) and equating to zero
the coefficients of the linearly independent coordinate 2-forms, one
obtains the conditions 
\begin{eqnarray}
f_{i1,z_{1}}=0,\quad f_{i1,z_{s}}=0,\quad\lambda\,f_{i1,z}+f_{i1,z_{2}}=0,\quad f_{i2,z_{s}}=0,\quad1\leq i\leq3,\quad s\geq3,\label{3.5.1}
\end{eqnarray}
and also the system 
\begin{eqnarray}
-f_{11,z}(A\,z_{3}+B)+{\displaystyle \sum_{i=0}^{2}z_{i+1}f_{12,z_{i}}+f_{32}f_{21}-f_{31}f_{22}=0,}\nonumber \\
-f_{21,z}(A\,z_{3}+B)+{\displaystyle \sum_{i=0}^{2}z_{i+1}f_{22,z_{i}}-f_{11}f_{32}+f_{31}f_{12}=0,}\label{(7.8)}\\
-f_{31,z}(A\,z_{3}+B)+{\displaystyle \sum_{i=0}^{2}z_{i+1}f_{32,z_{i}}+\delta(f_{21}f_{12}-f_{11}f_{22})=0.}\nonumber 
\end{eqnarray}
Then, from hypothesis \textbf{(i)-(ii)} one has that $f_{11}=\eta$
and there exist constants $\mu_{p}$, $\sigma_{p}$, $\alpha_{p}$
$\in$ $\mathbb{R}$, $2\leq p\leq3$, such that 
\begin{eqnarray}
f_{p1}=\mu_{p}\,z_{2}+\sigma_{p}\,z+\alpha_{p}\label{(i)-(ii)}
\end{eqnarray}
as in \eqref{fij}. Thus, the third equation of \eqref{3.5.1} together
with \eqref{(i)-(ii)} imply that 
\begin{eqnarray}
0=\lambda\,f_{p1,z}+f_{p1,z_{2}}=\lambda\,\sigma_{p}+\mu_{p},\quad p=2,3.\label{mu.sigma.cond}
\end{eqnarray}
Furthermore, by taking the derivative with respect to $z_{3}$ of
each equation in \eqref{(7.8)}, and taking into account (\ref{(i)-(ii)}),
one gets 
\begin{eqnarray}
f_{12,z_{2}}=0,\quad-\sigma_{p}\,A+f_{p2,z_{2}}=0,\quad p=2,3,\label{(7.9)}
\end{eqnarray}
and ($\ref{(7.8)}$) reduces to 
\begin{eqnarray}
{\displaystyle \sum_{i=0}^{1}z_{i+1}f_{12,z_{i}}+[\sigma_{2}(z-\lambda z_{2})+\alpha_{2}]\,f_{32}-[\sigma_{3}(z-\lambda z_{2})+\alpha_{3}]\,f_{22}=0,\label{(7.3)}}\\
-\sigma_{2}\,B+{\displaystyle \sum_{i=0}^{1}z_{i+1}f_{22,z_{i}}-\eta f_{32}+[\sigma_{3}(z-\lambda z_{2})+\alpha_{3}]\,f_{12}=0,}\label{(7.4)}\\
-\sigma_{3}\,B+{\displaystyle \sum_{i=0}^{1}z_{i+1}f_{32,z_{i}}+\delta\left\lbrace [\sigma_{2}(z-\lambda z_{2})+\alpha_{2}]f_{12}-\eta f_{22}\right\rbrace =0.}\label{(7.5)}
\end{eqnarray}
Now, from the first equation of ($\ref{(7.9)}$), one gets $f_{12}=f_{12}(z,z_{1})$.
Therefore, in view of \eqref{3.5.1} and \eqref{mu.sigma.cond} one
obtains the second and third equations of \eqref{7.2-2.3.1-cond_A}.
On the other hand, since $A^{2}+B^{2}\neq0$ on a nonempty open subset
and the structure equations \eqref{struttura} for equation \eqref{eq:C-1}
have been reduced to the second identity of \eqref{(7.9)} and equations
\eqref{(7.3)}-\eqref{(7.5)}, one necessarily has $\sigma_{2}^{2}+\sigma_{3}^{2}\neq0$.

Thus, the second identity of \eqref{(7.9)} now reads 
\begin{eqnarray*}
-\sigma_{2}\,A+f_{22,z_{2}}=0,\\
-\sigma_{3}\,A+f_{32,z_{2}}=0.
\end{eqnarray*}
and is equivalent to the system \eqref{A001}-\eqref{first.f22.f32},
where $\varphi=\varphi(z,z_{1})$ is a differentiable function. Indeed,
\eqref{A001} is obtained by adding the first equation multiplied
by $\sigma_{2}$ with the second one multiplied by $\sigma_{3}$,
whereas \eqref{first.f22.f32} readily follows by adding the first
equation multiplied by $\sigma_{3}$ with the second equation multiplied
by $-\sigma_{2}$ and further integrating the resulting equation with
respect to $z_{2}$. Then, by substituting \eqref{first.f22.f32}
into \eqref{(7.3)} one also gets \eqref{sec.f22.f32}.

Moreover, by multiplying \eqref{(7.4)} by $\sigma_{2}$ and adding
with \eqref{(7.5)} multiplied by $\sigma_{3}$ one obtains \eqref{B001}.
Analogously, by multiplying \eqref{(7.4)} by $\sigma_{3}$ and adding
with \eqref{(7.5)} multiplied by $-\sigma_{2}$, equation \eqref{first.f22.f32}
leads to \eqref{last.eq.f22.f32}. Finally, \eqref{7.5.1} follows
by $\omega_{1}\wedge\omega_{2}\neq0$.

The converse is a straightforward computation. This concludes the
proof of Lemma \ref{lema1}. 
\end{proof}
\medskip{}

\subsection{Proof of Theorem \ref{teo.clas.3.2}}

Following the notations used in Lemma \ref{lema1}, we will consider
here $\left(f_{31,z}\right)^{2}-\delta\left(f_{21,z}\right)^{2}=\sigma_{3}^{2}-\delta\sigma_{2}^{2}=0$,
with either $\eta=0$ or $\eta\neq0$.

\subsubsection{Case $\eta=0$}

By Lemma \ref{lema1} we need to determine the functions $f_{ij}$
that satisfy equations \eqref{fij}-\eqref{7.5.1}. First, in view
of $\sigma_{3}^{2}+\sigma_{2}^{2}\neq0$ (see Lemma \ref{lema1}),
by $\sigma_{3}^{2}-\delta\sigma_{2}^{2}=0$ we immediately have $\delta=1$
and $\sigma_{3}=\pm\sigma_{2}$. Then condition $\sigma_{3}^{2}+\sigma_{2}^{2}\neq0$
leads to $\sigma_{2}\neq0$ and, in view of \eqref{first.f22.f32},
one gets 
\begin{eqnarray}
f_{32}=\pm f_{22}-\dfrac{\varphi}{\sigma_{2}},\label{g32.teo3.2}
\end{eqnarray}
which replaced in \eqref{sec.f22.f32} gives 
\begin{eqnarray}
(\alpha_{3}\mp\alpha_{2})f_{22}={\displaystyle \sum_{i=0}^{1}z_{i+1}f_{12,z_{i}}-(z-\lambda z_{2})\,\varphi-\frac{\alpha_{2}}{\sigma_{2}}\varphi.}\label{4.14}
\end{eqnarray}
Furthermore, \eqref{last.eq.f22.f32} is equivalent to the system
of the following two equations 
\begin{eqnarray}
\varphi_{,z_{1}}=0,\quad z_{1}\,\varphi_{,z}\pm\sigma_{2}(\alpha_{3}\mp\alpha_{2})\,f_{12}=0.\label{4.15}
\end{eqnarray}

From now on our analysis will be divided into two cases:\textbf{ (I)}
$\alpha_{3}\mp\alpha_{2}=0$;\textbf{ (II)} $\alpha_{3}\mp\alpha_{2}\neq0$.

\textbf{Case (I) $\boxed{\alpha_{3}\mp\alpha_{2}=0}$.} In this case
\eqref{4.15} leads to $\varphi$ constant. Then \eqref{4.14} is
equivalent to $f_{12}=n$ constant and $\varphi=0$. On the other
hand \eqref{mu.sigma.cond} implies $n\neq0$. Thus, by replacing
$\varphi=0$ and $f_{22}:=\sigma_{2}\psi$, with $\psi=\psi(z,z_{1},z_{2})$
being a differentiable function, in \eqref{g32.teo3.2} and taking
into account \eqref{A001} and \eqref{B001} with $\eta=0$, one obtains
the partial differential equation 
\[
z_{t}-\lambda\,z_{2t}=D_{x}\psi\pm\,n\,(z-\lambda\,z_{2})\pm\frac{\alpha_{2}}{\sigma_{2}}\,n\,,
\]
with $n\,\sigma_{2}\neq0$, together with the corresponding functions
\[
\begin{array}{ll}
 & f_{11}=0,\vspace*{0.3cm}\\
 & f_{21}=\sigma_{2}(z-\lambda\,z_{2})+\alpha_{2},\vspace*{0.3cm}\\
 & f_{31}=\pm\,\sigma_{2}(z-\lambda\,z_{2})\pm\alpha_{2},
\end{array}\hspace*{1cm}\begin{array}{ll}
f_{12}=n,\vspace*{0.5cm}\\
f_{22}=\sigma_{2}\,\psi,\vspace*{0.5cm}\\
f_{32}=\pm\,\sigma_{2}\,\psi.
\end{array}
\]
Then, by performing a transformation of the form 
\[
t\mapsto\pm\,n\,t,\qquad z\mapsto\sigma_{2}\,\left(z+{\displaystyle \frac{\alpha_{2}}{\sigma_{2}}}\right),\qquad\psi\mapsto\pm{\displaystyle \frac{\sigma_{2}}{n}}\,\psi,
\]
one gets the differential equation describing \textbf{pss} given by
(I) and associated $1$-forms, with $\eta=0$, $a=0$, $b=1$.

\textbf{Case (II) $\boxed{\alpha_{3}\mp\alpha_{2}\neq0}$.} In this
case \eqref{4.15} leads to $\varphi(z,z_{1})=h$, with $h=h(z)$
a differentiable function, and 
\begin{eqnarray}
f_{12}=\mp\dfrac{h'}{\sigma_{2}(\alpha_{3}\mp\alpha_{2})}\,z_{1}.\label{4.16}
\end{eqnarray}
In particular, from \eqref{7.5.1} one gets that $h'\neq0$. Thus,
in view of \eqref{4.16}, \eqref{4.14},\eqref{g32.teo3.2}, \eqref{A001}
and \eqref{B001} one gets the partial differential equation

\begin{eqnarray*}
z_{t}-\lambda\,z_{2t}=\dfrac{\lambda\,h}{\sigma_{2}(\alpha_{3}\mp\alpha_{2})}z_{3}\mp\,D_{x}^{3}\left(\dfrac{h}{\sigma_{2}^{2}(\alpha_{3}\mp\alpha_{2})^{2}}\right)\hspace*{8cm}\\
\\
-2\left((\alpha_{3}\mp\alpha_{2})\sigma_{2}\,z\pm\dfrac{\alpha_{3}^{2}-\alpha_{2}^{2}}{2}-\lambda(\alpha_{3}\mp\alpha_{2})\sigma_{2}\,z_{2}\right)D_{x}\left(\dfrac{h}{\sigma_{2}^{2}(\alpha_{3}\mp\alpha_{2})^{2}}\right)-\frac{h}{\sigma_{2}(\alpha_{3}\mp\alpha_{2})}z_{1}.
\end{eqnarray*}
Furthermore, in view of \eqref{fij}, \eqref{g32.teo3.2}, \eqref{4.14}
and \eqref{4.16} one also gets the corresponding functions \vspace*{0.3cm}

\[
\begin{array}{ll}
 & f_{11}=0,\vspace*{0.3cm}\\
 & f_{21}=\sigma_{2}(z-\lambda z_{2})+\alpha_{2},\vspace*{0.3cm}\\
 & f_{31}=\pm\sigma_{2}(z-\lambda z_{2})+\alpha_{3},\vspace*{0.3cm}\\
 & f_{12}=\mp\dfrac{h'}{\sigma_{2}(\alpha_{3}\mp\alpha_{2})}\,z_{1},\vspace*{0.3cm}\\
 & f_{22}=\dfrac{1}{(\alpha_{3}\mp\alpha_{2})}\left(\mp\dfrac{h''}{\sigma_{2}(\alpha_{3}\mp\alpha_{2})}\,z_{1}^{2}\mp\dfrac{h'}{\sigma_{2}(\alpha_{3}\mp\alpha_{2})}\,z_{2}-(z-\lambda z_{2})\,h-\dfrac{\alpha_{2}}{\sigma_{2}}\,h\right),\vspace*{0.3cm}\\
 & f_{32}=\pm\dfrac{1}{(\alpha_{3}\mp\alpha_{2})}\left(\mp\dfrac{h''}{\sigma_{2}(\alpha_{3}\mp\alpha_{2})}\,z_{1}^{2}\mp\dfrac{h'}{\sigma_{2}(\alpha_{3}\mp\alpha_{2})}\,z_{2}-(z-\lambda z_{2})\,h\mp\dfrac{\alpha_{3}}{\sigma_{2}}\,h\right).
\end{array}\hspace*{1cm}
\]

\vspace*{0.3cm}

\noindent Then, by performing a transformation of the form 
\[
\begin{array}{l}
t\mapsto\left(\alpha_{3}\mp\alpha_{2}\right)\sigma_{2}\,t,\qquad z\mapsto\left(\alpha_{3}\mp\alpha_{2}\right)\sigma_{2}\,z\pm\left({\displaystyle \frac{\alpha_{3}^{2}-\alpha_{2}^{2}}{2}}\right),\qquad h\mapsto{\displaystyle \frac{h}{\sigma_{2}^{2}\left(\alpha_{3}\mp\alpha_{2}\right)^{2}}}\end{array},
\]
one gets the differential equation describing \textbf{pss} given by
(II) and associated $1$-forms, with $\eta=0$, $\alpha:=\alpha_{3}\mp\alpha_{2}$.

The converse is a straightforward computation.

\subsubsection{Case $\eta\protect\neq0$}

By Lemma \ref{lema1} we need to determine the functions $f_{ij}$
that satisfy equations \eqref{fij}-\eqref{7.5.1}. First, in view
of $\sigma_{3}^{2}-\delta\sigma_{2}^{2}=0$, we immediately have $\delta=1$
and $\sigma_{3}=\pm\sigma_{2}$. Then condition $\sigma_{3}^{2}+\sigma_{2}^{2}\neq0$
leads to $\sigma_{2}\neq0$, and in view of \eqref{first.f22.f32}
and \eqref{sec.f22.f32} one gets the following two conditions 
\begin{eqnarray}
 &  & f_{32}=\pm f_{22}-\dfrac{\varphi}{\sigma_{2}},\label{4.23}\\
 &  & (\alpha_{3}\mp\alpha_{2})\,f_{22}={\displaystyle \sum_{i=0}^{1}z_{i+1}\,f_{12,z_{i}}-(z-\lambda\,z_{2})\,\varphi-\dfrac{\alpha_{2}}{\sigma_{2}}\,\varphi.\label{4.24}}
\end{eqnarray}
Then, by substituting \eqref{4.23} in \eqref{last.eq.f22.f32} one
gets 
\begin{equation}
\mp\,\eta\,\varphi={\displaystyle \sum_{i=0}^{1}z_{i+1}\,\varphi_{,z_{i}}\pm\,\sigma_{2}(\alpha_{3}\mp\alpha_{2})\,f_{12},}\label{4.25}
\end{equation}
where, as shown above, $\varphi$ and $f_{12}$ are functions of $z$
and $z_{1}$. Thus, the derivative of \eqref{4.25} with respect to
$z_{2}$ leads to $\varphi_{,z_{1}}=0$ and replacing back to \eqref{4.25}
\begin{eqnarray}
\mp\,\eta\,h=z_{1}\,h_{,z}\pm\,\sigma_{2}\,(\alpha_{3}\mp\alpha_{2})\,f_{12},\label{4.26}
\end{eqnarray}
where $\varphi=h(z)$ is a differentiable function.

From now on our analysis will be divided into two cases:\textbf{ (I)}
$\alpha_{3}\mp\alpha_{2}=0$;\textbf{ (II)} $\alpha_{3}\mp\alpha_{2}\neq0$.

\textbf{Case (I) $\boxed{\alpha_{3}\mp\alpha_{2}=0}$.} In this case
\eqref{4.26} leads to $h=0$ and \eqref{4.24} to $f_{12}=n$ $\in$
$\mathbb{R}$. Then, by replacing \eqref{4.23} and $f_{22}:=\sigma_{2}\,\psi$,
with $\psi=\psi(z,z_{1},z_{2})$ being a differentiable function,
in \eqref{A001} and \eqref{B001}, one obtains the partial differential
equation 
\[
z_{t}-\lambda\,z_{2t}=D_{x}\psi\mp\eta\,\psi\pm\,n\,(z-\lambda\,z_{2})\pm\frac{n\,\alpha_{2}}{\sigma_{2}},
\]
together with the corresponding functions 
\[
\begin{array}{ll}
 & f_{11}=\eta,\vspace*{0.4cm}\\
 & f_{21}=\sigma_{2}(z-\lambda\,z_{2})+{\displaystyle \alpha_{2}},\vspace*{0.4cm}\\
 & f_{31}=\pm\,\sigma_{2}(z-\lambda\,z_{2})\pm\alpha_{2},
\end{array}\hspace*{1cm}\begin{array}{ll}
f_{12}=n,\vspace*{0.4cm}\\
f_{22}=\sigma_{2}\,\psi,\vspace*{0.4cm}\\
f_{32}=\pm\,\sigma_{2}\,\psi,
\end{array}
\]
where $\eta\psi-n\,(z-\lambda\,z_{xx})-{\displaystyle \frac{n\,\alpha_{2}}{\sigma_{2}}}\neq0$
due to \eqref{7.5.1}.

Then, by performing a transformation of the form 
\[
x\mapsto\mp\,\eta\,x,\qquad t\mapsto\pm\,n\,t,\qquad z\mapsto\mp{\displaystyle \frac{\sigma_{2}}{\eta}}\,\left(z+{\displaystyle \frac{\alpha_{2}}{\sigma_{2}}}\right),\qquad\psi\mapsto\pm{\displaystyle \frac{\sigma_{2}}{n}}\,\psi,\qquad\lambda\mapsto\eta^{2}\,\lambda
\]
with $\eta,n,\sigma_{2}\neq0$, we get the form given in (I) with
$\eta\neq0$, $a=1$ and $b=1$.

Analogously, by performing a transformation of the form 
\[
x\mapsto\mp\,\eta\,x,\qquad z\mapsto\mp\,{\displaystyle \frac{\sigma_{2}}{\eta}}\,\left(z+{\displaystyle \frac{\alpha_{2}}{\sigma_{2}}}\right),\qquad\psi\mapsto{\displaystyle \sigma_{2}}\,\psi,\qquad\lambda\mapsto\eta^{2}\,\lambda
\]
with $\eta,\sigma_{2}\neq0$ and $n=0$, one gets the differential
equation describing \textbf{pss} given by (I) and associated $1$-forms,
with $\eta\neq0$, $a=1$ and $b=0$.

\textbf{Case (II) $\boxed{\alpha_{3}\mp\alpha_{2}\neq0}$. }In this
case, in view of \eqref{4.26}, \eqref{4.24}, \eqref{4.23}, \eqref{A001}
and \eqref{B001} one gets the differential equation

\begin{eqnarray*}
z_{t}-\lambda\,z_{2t}=\dfrac{\lambda\,h}{\sigma_{2}(\alpha_{3}\mp\alpha_{2})}z_{3}\mp\,D_{x}^{3}\left(\dfrac{h}{\sigma_{2}^{2}(\alpha_{3}\mp\alpha_{2})^{2}}\right)\hspace*{8cm}\\
\\
-2\left((\alpha_{3}\mp\alpha_{2})\sigma_{2}\,z\pm\dfrac{\alpha_{3}^{2}-\alpha_{2}^{2}}{2}\mp\dfrac{\eta^{2}}{2}-\lambda(\alpha_{3}\mp\alpha_{2})\sigma_{2}\,z_{2}\right)D_{x}\left(\dfrac{h}{\sigma_{2}^{2}(\alpha_{3}\mp\alpha_{2})^{2}}\right)-\frac{h}{\sigma_{2}(\alpha_{3}\mp\alpha_{2})}z_{1},
\end{eqnarray*}
together with the corresponding functions 
\[
\begin{array}{ll}
 & f_{11}=\eta,\vspace*{0.3cm}\\
 & f_{21}=\sigma_{2}(z-\lambda z_{2})+\alpha_{2},\vspace*{0.4cm}\\
 & f_{31}=\pm\sigma_{2}(z-\lambda z_{2})+\alpha_{3},\vspace*{0.4cm}\\
 & f_{12}=\mp\dfrac{h'}{\sigma_{2}(\alpha_{3}\mp\alpha_{2})}\,z_{1}-\dfrac{\eta\,h}{\sigma_{2}(\alpha_{3}\mp\alpha_{2})},\vspace*{0.4cm}\\
 & f_{22}=\dfrac{1}{(\alpha_{3}\mp\alpha_{2})}\left(\mp\dfrac{h''}{\sigma_{2}(\alpha_{3}\mp\alpha_{2})}\,z_{1}^{2}\mp\dfrac{h'}{\sigma_{2}(\alpha_{3}\mp\alpha_{2})}\,z_{2}-\dfrac{\eta\,h'}{\sigma_{2}(\alpha_{3}\mp\alpha_{2})}\,z_{1}-(z-\lambda z_{2})\,h-\dfrac{\alpha_{2}}{\sigma_{2}}\,h\right),\vspace*{0.4cm}\\
 & f_{32}=\pm\dfrac{1}{(\alpha_{3}\mp\alpha_{2})}\left(\mp\dfrac{h''}{\sigma_{2}(\alpha_{3}\mp\alpha_{2})}\,z_{1}^{2}\mp\dfrac{h'}{\sigma_{2}(\alpha_{3}\mp\alpha_{2})}\,z_{2}-\dfrac{\eta\,h'}{\sigma_{2}(\alpha_{3}\mp\alpha_{2})}\,z_{1}-(z-\lambda z_{2})\,h\mp\dfrac{\alpha_{3}}{\sigma_{2}}\,h\right),
\end{array}\hspace*{1cm}
\]
where $h'\neq0$ due to \eqref{7.5.1}.

Then, by performing a transformation of the form 
\[
\begin{array}{l}
t\mapsto\left(\alpha_{3}\mp\alpha_{2}\right)\sigma_{2}\,t,\qquad z\mapsto\left(\alpha_{3}\mp\alpha_{2}\right)\sigma_{2}\,z\pm\left({\displaystyle \frac{\alpha_{3}^{2}-\alpha_{2}^{2}}{2}}\right)\mp{\displaystyle \frac{{\displaystyle \eta}^{2}}{2}},\qquad h\mapsto{\displaystyle \frac{h}{\sigma_{2}^{2}\left(\alpha_{3}\mp\alpha_{2}\right)^{2}}}\end{array}
\]
with $\alpha:=\alpha_{3}\mp\alpha_{2}$, one gets the differential
equation describing \textbf{pss} given by (II) and associated $1$-forms.

The converse is a straightforward computation.

\subsection{Proof of Theorem \ref{teo.clas.3.3}}

Following the notations used in Lemma \ref{lema1}, we will consider
here $\left(f_{31,z}\right)^{2}-\delta\left(f_{21,z}\right)^{2}=\sigma_{3}^{2}-\delta\sigma_{2}^{2}\neq0$,
with $\eta=0$. Thus we need to determine the functions $f_{ij}$
that satisfy (\ref{fij})-(\ref{7.5.1}). First, from equation \eqref{last.eq.f22.f32}
one gets the following two equations 
\begin{eqnarray}
f_{12}=-\dfrac{z_{1}}{\gamma\,z+m_{1}}\,\varphi_{,z},\label{4.17}
\end{eqnarray}

\begin{eqnarray}
(\gamma\,z+m_{1})\,\varphi_{,z_{1}}+\lambda\,\gamma\,z_{1}\,\varphi_{,z}=0,\label{4.18}
\end{eqnarray}
where 
\begin{equation}
\gamma:=\sigma_{3}^{2}-\delta\,\sigma_{2}^{2},\qquad m_{1}:=\sigma_{3}\alpha_{3}-\delta\sigma_{2}\alpha_{2}\label{eq:gamma-m1-th3.2}
\end{equation}
and $\gamma\,z+m_{1}\neq0$, in view of $\gamma\neq0$.

One can distinguish two cases: \textbf{(a)} $-\sigma_{3}\alpha_{2}+\sigma_{2}\alpha_{3}=0$;
\textbf{(b)} $-\sigma_{3}\alpha_{2}+\sigma_{2}\alpha_{3}\neq0$.

\textbf{(a)} In this case $\sigma_{2}\neq0$ . Indeed, by $\sigma_{2}^{2}+\sigma_{3}^{2}\neq0$
and $-\sigma_{3}\alpha_{2}+\sigma_{2}\alpha_{3}=0$, if $\sigma_{2}=0$
then $\alpha_{2}=0$, which contradicts non-degeneracy condition \eqref{7.5.1}.
Thus, $\sigma_{2}\neq0$ and $\alpha_{3}=\sigma_{3}\alpha_{2}/\sigma_{2}$.
Therefore by \eqref{first.f22.f32} and \eqref{sec.f22.f32} one gets
\begin{eqnarray}
\dfrac{\alpha_{2}}{\sigma_{2}}\varphi=\dfrac{\alpha_{2}}{\sigma_{2}}\left(\sigma_{3}f_{22}-\sigma_{2}f_{32}\right)=\alpha_{3}\,f_{22}-\alpha_{2}\,f_{32} & = & {\displaystyle \sum_{i=0}^{1}z_{i+1}f_{12,z_{i}}-(z-\lambda z_{2})\,\varphi.}\label{eq:aux1}
\end{eqnarray}
Hence, by differentiating \eqref{eq:aux1} with respect to $z_{2}$,
one has 
\begin{eqnarray}
f_{12,z_{1}}+\lambda\,\varphi=0,\qquad f_{12,z}-\dfrac{1}{z_{1}}\left(z+\dfrac{\alpha_{2}}{\sigma_{2}}\right)\varphi=0\label{4.19}
\end{eqnarray}
and the compatibility condition $f_{12,z\,z_{1}}=f_{12,z_{1}\,z}$
reads 
\begin{eqnarray*}
z_{1}\,\left[\left(z+\frac{\alpha_{2}}{\sigma_{2}}\right)\,\varphi_{,z_{1}}+\lambda\,z_{1}\,\varphi_{,z}\right]-\left(z+\frac{\alpha_{2}}{\sigma_{2}}\right)\varphi=0.
\end{eqnarray*}
On the other hand $\alpha_{3}=\sigma_{3}\alpha_{2}/\sigma_{2}$ entails
that $m_{1}=\alpha_{2}\,\gamma/\sigma_{2}$, and \eqref{4.18} leads
to 
\begin{eqnarray}
\left(z+\frac{\alpha_{2}}{\sigma_{2}}\right)\,\varphi_{,z_{1}}+\lambda\,z_{1}\,\varphi_{,z}=0.\label{4.20}
\end{eqnarray}
Thus $\varphi=0$ and by \eqref{4.17} one gets $f_{12}=0$, which
contradicts \eqref{7.5.1}. Hence the case $-\sigma_{3}\alpha_{2}+\sigma_{2}\alpha_{3}=0$
is not allowed.

\textbf{(b)} In this case \eqref{first.f22.f32} and \eqref{sec.f22.f32}
can be solved with respect to $f_{22}$ and $f_{32}$ as 
\begin{eqnarray}
 &  & f_{22}=\frac{1}{m_{2}}\left[-\alpha_{2}\,\varphi+\sigma_{2}\,{\displaystyle \sum_{i=0}^{1}z_{i+1}f_{12,z_{i}}-\sigma_{2}\,(z-\lambda z_{2})\,\varphi}\right],\label{4.21}\\
 &  & f_{32}=\frac{1}{m_{2}}\left[-\alpha_{3}\,\varphi+\sigma_{3}\,{\displaystyle \sum_{i=0}^{1}z_{i+1}f_{12,z_{i}}-\sigma_{3}\,(z-\lambda z_{2})\,\varphi}\right],\label{4.22}
\end{eqnarray}
where $f_{12}$ is given by \eqref{4.17}, $\varphi$ satisfies \eqref{4.18}
and 
\begin{equation}
m_{2}:=\sigma_{2}\alpha_{3}-\sigma_{3}\alpha_{2}.\label{eq:m2_th32}
\end{equation}

Hence, by substituting \eqref{4.21} and \eqref{4.22} in \eqref{A001}
and \eqref{B001}, and using \eqref{4.17} and \eqref{4.18}, we get
the partial differential equation 
\begin{center}
\begin{eqnarray}
z_{t}-\lambda\,z_{2t}=-D_{x}^{2}\left[z_{1}\left(\dfrac{\varphi_{,z}}{\gamma\,z+m_{1}}\right)\right]-D_{x}\left[(z-\lambda\,z_{2})\,\varphi\right]-\frac{m_{1}}{\gamma}\,\varphi_{z_{1}}\,z_{2}-\frac{1}{\gamma}\left(m_{1}-\frac{\delta m_{2}^{2}}{\gamma z+m_{1}}\right)\,\varphi_{,z}\,z_{1},\label{eq..iiiii3}
\end{eqnarray}
\par\end{center}

\vspace*{0.3cm}
\noindent where $m_{2}\gamma\neq0$ and, in order to simplify the obtained
partial differential equation, $\varphi$ was rearranged as $\varphi\mapsto\varphi/m_{2}$.
Accordingly, in view of \eqref{4.17} and \eqref{4.21}-\eqref{4.22}
one gets the corresponding functions \vspace*{0.3cm}

\noindent 
\[
\begin{array}{ll}
 & f_{11}=0,\vspace*{0.3cm}\\
 & f_{21}=\sigma\,(z-\lambda\,z_{2})+\dfrac{\sigma\,m_{1}\mp\,m_{2}\,\sqrt{\gamma+\delta\sigma^{2}}}{\gamma},\qquad\gamma+\delta\sigma^{2}\geq0,\qquad\sigma\in\mathbb{R},\vspace*{0.3cm}\\
 & f_{31}=\pm\sqrt{\gamma+\delta\sigma^{2}}\,\,(z-\lambda\,z_{2})-\dfrac{\delta\sigma\,m_{2}\mp\,m_{1}\sqrt{\gamma+\delta\sigma^{2}}}{\gamma},\vspace*{0.3cm}\\
 & f_{12}=-\dfrac{m_{2}\,\varphi_{,z}}{\gamma\,z+m_{1}}\,z_{1},\vspace*{0.3cm}\\
 & f_{22}=-\sigma\,D_{x}\left[z_{1}\left(\dfrac{\varphi_{,z}}{\gamma\,z+m_{1}}\right)\right]-\sigma\,(z-\lambda\,z_{2})\,\varphi-\dfrac{\sigma\,m_{1}\mp\,m_{2}\,\sqrt{\gamma+\delta\sigma^{2}}}{\gamma}\,\varphi\,,\vspace*{0.3cm}\\
 & f_{32}=\mp\,\sqrt{\gamma+\delta\sigma^{2}}\,\,D_{x}\left[z_{1}\left(\dfrac{\varphi_{,z}}{\gamma\,z+m_{1}}\right)\right]\mp\,\sqrt{\gamma+\delta\sigma^{2}}\,(z-\lambda\,z_{2})\,\varphi+\dfrac{\delta\sigma\,m_{2}\mp\,m_{1}\sqrt{\gamma+\delta\sigma^{2}}}{\gamma}\,\varphi\,,
\end{array}\hspace*{1cm}
\]

\vspace*{0.3cm}
\noindent where $\sigma:=\sigma_{2}$, $\gamma+\delta\,\sigma^{2}\geq0$, $\varphi=\varphi(z,z_{1})$
is a differentiable function satisfying $\varphi_{,z}\neq0$ and \eqref{4.18};
moreover we took account of \eqref{eq:gamma-m1-th3.2} and \eqref{eq:m2_th32}
by setting $\sigma_{3}=\pm\,\sqrt{\gamma+\delta\,\sigma^{2}}$, $\alpha_{2}=(\sigma\,m_{1}-\sigma_{3}\,m_{2})/\gamma$
and $\alpha_{3}=(\sigma_{3}\,m_{1}-\delta\,\sigma\,m_{2})/\gamma$.

\noindent Now, by performing the transformation $\{z\mapsto\gamma z+m_{1}\}$
one can see that \eqref{4.18} reduces to $z\varphi_{,z_{1}}+\lambda z_{1}\varphi_{,z}=0$.
Hence without loss of generality one can always put $m_{1}=0$ and
take $\varphi=\varphi(\lambda z_{1}^{2}-z^{2})$. Further simplifications
follow by observing that $\varphi_{z}/z=-2\varphi'$ and setting $r:=-m_{2}$
equation \eqref{eq..iiiii3} reduces to the desired differential equation
describing \textbf{pss} and \textbf{ss}, together with the associated
$1$-forms.

The converse is a straightforward computation.


\subsection{Proof of Theorem \ref{teo.clas.3.5}}

Following the notations used in Lemma \ref{lema1}, we will consider
here $\left(f_{31,z}\right)^{2}-\delta\left(f_{21,z}\right)^{2}=\sigma_{3}^{2}-\delta\sigma_{2}^{2}\neq0$,
with $\eta\neq0$. Thus we need to determine the functions $f_{ij}$
that satisfy (\ref{fij})-(\ref{7.5.1}). First, from \eqref{first.f22.f32}
and \eqref{last.eq.f22.f32} one gets 
\begin{eqnarray}
 &  & f_{22}=\dfrac{\sigma_{3}}{\gamma}\,\varphi+\dfrac{\sigma_{2}}{\eta\gamma}{\displaystyle \sum_{i=0}^{1}z_{i+1}\,\varphi_{,z_{i}}+\dfrac{\sigma_{2}}{\eta\gamma}\left[\gamma\,(z-\lambda z_{2})+m_{1}\right]\,f_{12}\,,\label{4.27}}\\
 &  & f_{32}=\dfrac{\delta\sigma_{2}}{\gamma}\,\varphi+\dfrac{\sigma_{3}}{\eta\gamma}{\displaystyle \sum_{i=0}^{1}z_{i+1}\,\varphi_{,z_{i}}+\dfrac{\sigma_{3}}{\eta\gamma}\left[\gamma\,(z-\lambda z_{2})+m_{1}\right]\,f_{12}\,,\label{4.28}}
\end{eqnarray}
where $\varphi^{2}+f_{12}^{2}\neq0$, in view of (\ref{7.5.1}), and
$\gamma,m_{1}$ are the constants defined in \eqref{eq:gamma-m1-th3.2}.

Then, by substituting \eqref{4.27} and \eqref{4.28} in \eqref{sec.f22.f32},
and differentiating with respect to $z_{2}$ one gets the following
two equations 
\begin{eqnarray}
\ell_{1,z_{1}}+\lambda\,\gamma\,\ell_{2}=0,\label{4.29}\\
\nonumber \\
z_{1}\,\ell_{1,z}-(\gamma\,z+m_{1})\,\ell_{2}=0,\label{4.30}
\end{eqnarray}
where 
\begin{equation}
\ell_{1}:=\eta\,\gamma\,f_{12}-m_{2}\varphi,\qquad\ell_{2}:=m_{2}\,f_{12}+\eta\,\varphi\label{eq:ell_th3.3}
\end{equation}
and $m_{2}$ is the constant defined in \eqref{eq:m2_th32}.

Now, since $\gamma\neq0$, one has $\gamma\,z+m_{1}\neq0$ and in
view of \eqref{4.30} equation \eqref{4.29} leads to the following
partial differential equation 
\begin{eqnarray}
(\gamma\,z+m_{1})\,\ell_{1,z_{1}}+\lambda\,\gamma\,z_{1}\,\ell_{1,z}=0.\label{4.31}
\end{eqnarray}

On the other hand, by using above expression of $\ell_{2}$ together
with \eqref{4.30}, one obtains 
\begin{eqnarray}
\varphi=-\dfrac{m_{2}}{\eta}\,f_{12}+\dfrac{z_{1}}{\eta\,(\gamma\,z+m_{1})}\,\ell_{1,z},\label{4.32}
\end{eqnarray}
which replaced in the expression of $\ell_{1}$ leads to 
\begin{eqnarray}
(\gamma\,\eta^{2}+m_{2}^{2})\,f_{12}=\eta\,\ell_{1}+\frac{m_{2}\,z_{1}}{\gamma\,z+m_{1}}\,\ell_{1,z}.\label{4.33}
\end{eqnarray}
From now on we will distinguish the cases:\textbf{ (I)} $\gamma\,\eta^{2}+m_{2}^{2}=0$;
\textbf{(II) }$\gamma\,\eta^{2}+m_{2}^{2}\neq0$.

\textbf{Case (I)} $\boxed{\gamma\,\eta^{2}+m_{2}^{2}=0}$. First we
notice that in this case $\gamma<0$ and hence by \eqref{eq:gamma-m1-th3.2}
it follows that $\delta=1$. Moreover, by $\gamma\eta\neq0$ one has
$m_{2}\neq0$, and we state that $\ell_{1}=0$. Indeed, if $\ell_{1}\neq0$
one could rewrite \eqref{4.33} in the form 
\[
\dfrac{\ell_{1,z}}{\ell_{1}}=-\dfrac{\eta\,(\gamma\,z+m_{1})}{m_{2}\,z_{1}}.
\]
Then, by substituting the solution of the last equation in \eqref{4.31}
one would get $\gamma=0$, in contradiction with the assumption $\gamma=\sigma_{3}^{2}-\delta\sigma_{2}^{2}\neq0$.
Thus, $\ell_{1}=0$ and \eqref{4.31} is satisfied. On the other hand,
by \eqref{4.32} one has $f_{12}=-\eta\,\varphi/m_{2}$ and, by \eqref{4.27}
and \eqref{4.28}, one gets $f_{22}$ and $f_{32}$ as functions of
$\varphi$. Therefore, by replacing such functions in \eqref{A001}
and \eqref{B001}, one obtains the partial differential equation

\begin{eqnarray}
z_{t}-\lambda\,z_{2t}=\dfrac{1}{\eta\,\gamma}\,D_{x}^{2}\left(\varphi\right)-D_{x}\left[\dfrac{\varphi}{m_{2}\,\gamma}(\gamma\,z+m_{1}-\lambda\,\gamma\,z_{2})\right]\,,\label{eq55.t3}
\end{eqnarray}
together with the corresponding functions 
\[
\begin{array}{ll}
 & f_{11}=\eta,\vspace*{0.4cm}\\
 & f_{21}=\sigma\,(z-\lambda\,z_{2})+\dfrac{\sigma\,m_{1}\mp\,m_{2}\,\sqrt{\gamma+\sigma^{2}}}{\gamma},\qquad\gamma+\sigma^{2}\geq0,\qquad\sigma\in\mathbb{R}\setminus\{0\},\vspace*{0.4cm}\\
 & f_{31}=\pm\sqrt{\gamma+\sigma^{2}}\,\,(z-\lambda\,z_{2})-\dfrac{\sigma\,m_{2}\mp\,m_{1}\sqrt{\gamma+\sigma^{2}}}{\gamma},\vspace*{0.4cm}\\
 & f_{12}=-\dfrac{\eta}{m_{2}}\,\varphi,\vspace*{0.4cm}\\
 & f_{22}=\pm\dfrac{\sqrt{\gamma+\sigma^{2}}}{\gamma}\,\varphi+\dfrac{\sigma}{\eta\gamma}\,D_{x}(\varphi)-\dfrac{\sigma}{m_{2}\,\gamma}\left[\gamma\,(z-\lambda z_{2})+m_{1}\right]\varphi\,,\vspace*{0.4cm}\\
 & f_{32}=\dfrac{\sigma}{\gamma}\,\varphi\pm\dfrac{\sqrt{\gamma+\sigma^{2}}}{\eta\gamma}\,D_{x}(\varphi)\mp\dfrac{\sqrt{\gamma+\sigma^{2}}}{m_{2}\,\gamma}\left[\gamma\,(z-\lambda z_{2})+m_{1}\right]\varphi\,,
\end{array}\hspace*{1cm}
\]

\noindent where $\sigma:=\sigma_{2}$, $m_{2}\neq0$, $\gamma<0$, $\gamma+\sigma^{2}\geq0$,
$\varphi=\varphi(z,z_{1})$ is a differentiable function, that in
view of \eqref{7.5.1} satisfy $D_{x}\varphi\neq0$, and we took account
of \eqref{eq:gamma-m1-th3.2} and \eqref{eq:m2_th32} by setting $\sigma_{3}=\pm\sqrt{\gamma+\sigma^{2}}$,
$\alpha_{2}=(\sigma_{2}m_{1}-\sigma_{3}m_{2})/\gamma$, $\alpha_{3}=(\sigma_{3}m_{1}-\delta\sigma_{2}m_{2})/\gamma$.

Then, by suitably adapting $z$, $\varphi$, $m_{2}$ according to
\[
z\mapsto\eta\left(\gamma z+m_{1}\right)/m_{2},\qquad\varphi\mapsto\varphi/m_{2},
\]
and introducing $\rho:=\sigma\,\eta/m_{2}$, one gets the differential
equation describing \textbf{pss} given by (I), together with the associated
$1$-forms.

\noindent \textbf{Case (II)} $\boxed{\gamma\,\eta^{2}+m_{2}^{2}\neq0}$.
From \eqref{4.33} and \eqref{4.32} one obtains $f_{12}$ and $\varphi$
in tems of $\ell_{1}$ and $\ell_{1,z}$, where $\ell_{1}$ is solution
of \eqref{4.31}. Then, by substituting \eqref{4.27} and \eqref{4.28}
in \eqref{A001} and \eqref{B001}, one has the following partial
differential equation 
\begin{eqnarray}
z_{t}-\lambda\,z_{2t}=\dfrac{1}{\eta\,\gamma}D_{x}^{2}(\varphi)+\frac{1}{\eta\,\gamma}D_{x}\left[\phi\,(\gamma\,z+m_{1}-\lambda\,\gamma\,z_{2})\right]-\dfrac{\delta\,m_{2}}{\gamma}\,\phi-\dfrac{\delta\,\eta\,\varphi}{\gamma},\label{eq56_prof}
\end{eqnarray}
where $\varphi=\varphi(z,z_{1})$ and $\phi=\phi(z,z_{1})$ are 
\begin{eqnarray*}
 &  & \varphi=\dfrac{\gamma\,\eta\,\ell_{,z}\,z_{1}}{(\gamma\,\eta^{2}+m_{2}^{2})(\gamma\,z+m_{1})}-\dfrac{m_{2}\,\ell}{\gamma\,\eta^{2}+m_{2}^{2}}\,,\\
\\
 &  & \phi=\dfrac{m_{2}\,\ell_{,z}\,z_{1}}{(\gamma\,\eta^{2}+m_{2}^{2})(\gamma\,z+m_{1})}+\dfrac{\eta\,\ell}{\gamma\,\eta^{2}+m_{2}^{2}}\,,\\
\end{eqnarray*}
and $\gamma(\gamma\,\eta^{2}+m_{2}^{2})\neq0$. Also, from \eqref{fij},
\eqref{4.27}, \eqref{4.28}, \eqref{4.31}, \eqref{4.32} and \eqref{4.33},
one gets the corresponding functions

\[
\begin{array}{ll}
 & f_{11}=\eta,\vspace*{0.3cm}\\
 & f_{21}=\sigma\,(z-\lambda\,z_{2})+\dfrac{\sigma\,m_{1}\mp\,m_{2}\,\sqrt{\gamma+\delta\sigma^{2}}}{\gamma},\qquad\gamma+\delta\sigma^{2}\geq0,\qquad\sigma\in\mathbb{R},\vspace*{0.4cm}\\
 & f_{31}=\pm\sqrt{\gamma+\delta\sigma^{2}}\,\,(z-\lambda\,z_{2})-\dfrac{\delta\sigma\,m_{2}\mp\,m_{1}\sqrt{\gamma+\delta\sigma^{2}}}{\gamma},\vspace*{0.4cm}\\
 & f_{12}=\dfrac{m_{2}\,\ell_{,z}\,z_{1}}{(\gamma\,\eta^{2}+m_{2}^{2})(\gamma\,z+m_{1})}+\dfrac{\eta\,\ell}{\gamma\,\eta^{2}+m_{2}^{2}}\,,\vspace*{0.4cm}\\
 & f_{22}=\pm\dfrac{\sqrt{\gamma+\delta\sigma^{2}}}{\gamma}\,\varphi+\dfrac{\sigma}{\eta\gamma}D_{x}(\varphi)+\dfrac{\sigma}{\eta\gamma}\left[\gamma\,(z-\lambda z_{2})+m_{1}\right]\,\phi\,,\vspace*{0.4cm}\\
 & f_{32}=\dfrac{\delta\sigma}{\gamma}\,\varphi\pm\dfrac{\sqrt{\gamma+\delta\sigma^{2}}}{\eta\gamma}D_{x}(\varphi)\pm\dfrac{\sqrt{\gamma+\delta\sigma^{2}}}{\eta\gamma}\left[\gamma\,(z-\lambda z_{2})+m_{1}\right]\,\phi\,,
\end{array}\hspace*{1cm}
\]

\vspace*{0.3cm}
\noindent where $\sigma:=\sigma_{2}$, $r:=m_{2}$ $\gamma+\delta\,\sigma^{2}\geq0$,
$\ell:=\ell_{1}$ is a nonzero differentiable function of $(z,z_{1})$
satisfying the differential equation \eqref{4.31}, and we took account
of \eqref{eq:gamma-m1-th3.2} and \eqref{eq:m2_th32} by setting $\sigma_{3}=\pm\,\sqrt{\gamma+\delta\,\sigma^{2}}$,
$\alpha_{2}=(\sigma\,m_{1}-\sigma_{3}\,m_{2})/\gamma$ and $\alpha_{3}=(\sigma_{3}\,m_{1}-\delta\,\sigma\,m_{2})/\gamma$.
In addition, $\eta\,f_{22}-\phi\,f_{21}\neq0$ due to nondegeneracy
condition \eqref{7.5.1}.

\noindent Now, by performing a transformation $z\mapsto\gamma z+m_{1}$,
equation \eqref{4.31} reduces to $z\ell_{,z_{1}}+\lambda z_{1}\ell_{,z}=0$.
Hence without loss of generality one can always put $m_{1}=0$ and
take $\ell=\ell(\lambda z_{1}^{2}-z^{2})$. Further simplifications
follow by observing that $\ell_{z}/z=-2\ell'$ and \eqref{eq56_prof}
reduces to the desired differential equation describing \textbf{pss}
or \textbf{ss} given by (II), together with the associated $1$-forms.
In particular, $\ell'\neq0$ due to $\omega_{1}\wedge\omega_{2}\neq0$.

\noindent The converse is a straightforward computation.

\section{Proof of the main results for equations of the form $z_{2t}=A\,z_{3}+B$
\label{sec4*}}

This section is devoted to the proofs of Theorems \ref{teo.clas.3.2*},
\ref{teo.clas.3.3*} and \ref{teo.clas.3.5*}.

\subsection{Preliminary lemma}

We start with a lemma that gives a convenient characterization of
equations \eqref{eq:C-2} describing \textbf{pss} or \textbf{ss} ($\delta=1$
or $\delta=-1,$ respectively) under assumptions \textbf{(i-ii)}.

\medskip{}

\begin{lem}
\label{lema1*} A partial differential equation of the form \eqref{eq:C-2}
describes \textbf{pss} or \textbf{ss} ($\delta=1$ or $\delta=-1,$
respectively), with corresponding functions $f_{ij}=f_{ij}(z,z_{1},...,z_{k})$
satisfying \textbf{(i-ii)}, if and only if 
\begin{equation}
f_{11}=\eta,\quad\quad f_{p1}=\mu_{p}\,z_{2}+\alpha_{p},\quad2\leq p\leq3,\label{fij*}
\end{equation}
\begin{equation}
\begin{array}{lcccc}
f_{12}=f_{12}(z,z_{1}),\qquad & f_{p2}=f_{p2}(z,z_{1},z_{2}),\quad & p=2,3,\end{array}\label{7.2-2.3.1-cond_A*}
\end{equation}
\begin{eqnarray}
(\mu_{2}^{2}+\mu_{3}^{2})\,A=\mu_{2}\,f_{22,z_{2}}+\mu_{3}\,f_{32,z_{2}}\,,\label{A3D*}
\end{eqnarray}
\begin{eqnarray}
\mu_{3}f_{22}-\mu_{2}f_{32} & = & \varphi,\label{first.f22.f32*}\\
\alpha_{3}\,f_{22}-\alpha_{2}\,f_{32} & = & {\displaystyle \sum_{i=0}^{1}z_{i+1}f_{12,z_{i}}-z_{2}\,\varphi,\label{sec.f22.f32*}}
\end{eqnarray}

\begin{eqnarray}
(\mu_{2}^{2}+\mu_{3}^{2})\,B={\displaystyle \sum_{i=0}^{1}z_{i+1}(\mu_{2}\,f_{22}+\mu_{3}\,f_{32})_{,z_{i}}-\eta(\mu_{2}\,f_{32}+\delta\mu_{3}\,f_{22})}{\displaystyle \qquad\qquad\qquad\nonumber}\\
+[\mu_{2}\mu_{3}(1+\delta)z_{2}+\mu_{2}\alpha_{3}+\delta\mu_{3}\alpha_{2}]\,f_{12},\label{B3D*}
\end{eqnarray}

\begin{equation}
\eta(\mu_{3}\,f_{32}-\delta\mu_{2}\,f_{22})={\displaystyle \sum_{i=0}^{1}z_{i+1}\,\varphi_{,z_{i}}+[(\mu_{3}^{2}-\delta\mu_{2}^{2})z_{2}+\mu_{3}\alpha_{3}-\delta\mu_{2}\alpha_{2}]\,f_{12},}\label{last.eq.f22.f32*}
\end{equation}
where $\eta,\mu_{p},\alpha_{p}\in\mathbb{R}$, $\mu_{2}^{2}+\mu_{3}^{2}\neq0$,
and $\varphi=\varphi(z,z_{1})$ is a real and differentiable function
of $z$ and $z_{1}$. Furthermore, 
\begin{eqnarray}
\eta f_{22}-(\mu_{2}z_{2}+\alpha_{2})\,f_{12}\neq0.\label{7.5.1*}
\end{eqnarray}
\end{lem}

\begin{proof}
Let $z=z(x,t)$ be a function satisfying ($\ref{eq:C-2}$). Then
\begin{eqnarray}
 &  & dz_{2}\wedge dx=-(A\,z_{3}+B)dx\wedge dt,\label{(7.6)*}\\
 &  & dz_{i}\wedge dt=z_{i+1}dx\wedge dt,\quad0\leq i\leq m.\nonumber 
\end{eqnarray}
Since the 1-forms $\omega_{i}$ satisfy the structure equations ($\ref{struttura}$),
one gets 
\begin{eqnarray}
df_{11}\wedge dx+df_{12}\wedge dt+(f_{32}f_{21}-f_{31}f_{22})dx\wedge dt=0,\nonumber \\
df_{21}\wedge dx+df_{22}\wedge dt+(-f_{11}f_{32}+f_{31}f_{12})dx\wedge dt=0,\label{(7.7-a)*}\\
df_{31}\wedge dx+df_{32}\wedge dt+\delta(f_{21}f_{12}-f_{11}f_{22})dx\wedge dt=0.\nonumber 
\end{eqnarray}
Then, by substituting 
\[
df_{ij}={\displaystyle \sum_{k=0}^{m}f_{ij,z_{k}}dz_{k},}
\]
in ($\ref{(7.7-a)*}$), using ($\ref{(7.6)*}$) and equating to zero
the coefficients of the linearly independent coordinate 2-forms, one
obtains the conditions 
\begin{eqnarray}
f_{i1,z_{r}}=0,\quad r\geq0,r\neq2,\quad\quad\quad f_{i2,z_{s}}=0,\quad1\leq i\leq3,\quad s\geq3,\label{3.5.1*}
\end{eqnarray}
and also the system 
\begin{eqnarray}
-f_{11,z_{2}}(A\,z_{3}+B)+{\displaystyle \sum_{i=0}^{2}z_{i+1}f_{12,z_{i}}+f_{32}f_{21}-f_{31}f_{22}=0,}\nonumber \\
-f_{21,z_{2}}(A\,z_{3}+B)+{\displaystyle \sum_{i=0}^{2}z_{i+1}f_{22,z_{i}}-f_{11}f_{32}+f_{31}f_{12}=0,}\label{(7.8)*}\\
-f_{31,z_{2}}(A\,z_{3}+B)+{\displaystyle \sum_{i=0}^{2}z_{i+1}f_{32,z_{i}}+\delta(f_{21}f_{12}-f_{11}f_{22})=0.}\nonumber 
\end{eqnarray}
Then, in view of \textbf{(i)-(ii),} the first equation of \eqref{3.5.1*}
leads to \eqref{fij*}. Moreover, \eqref{(7.8)*} reduces to 
\begin{eqnarray}
f_{12,z_{2}}=0,\quad-\mu_{p}\,A+f_{p2,z_{2}}=0,\quad p=2,3,\label{(7.9)*}
\end{eqnarray}
and the further equations 
\begin{eqnarray}
{\displaystyle \sum_{i=0}^{1}z_{i+1}f_{12,z_{i}}+(\mu_{2}z_{2}+\alpha_{2})\,f_{32}-(\mu_{3}z_{2}+\alpha_{3})\,f_{22}=0,\label{(7.3)*}}\\
-\mu_{2}\,B+{\displaystyle \sum_{i=0}^{1}z_{i+1}f_{22,z_{i}}-\eta f_{32}+(\mu_{3}z_{2}+\alpha_{3})\,f_{12}=0,}\label{(7.4)*}\\
-\mu_{3}\,B+{\displaystyle \sum_{i=0}^{1}z_{i+1}f_{32,z_{i}}+\delta[(\mu_{2}z_{2}+\alpha_{2})f_{12}-\eta f_{22}]=0.}\label{(7.5)*}
\end{eqnarray}

Now, from the first equation of ($\ref{(7.9)*}$) one has $f_{12}=f_{12}(z,z_{1})$,
which is the first equation of \eqref{7.2-2.3.1-cond_A*}. Furthermore,
the second equation of \eqref{7.2-2.3.1-cond_A*} follows by second
equation of \eqref{3.5.1*}. On the other hand, since $A^{2}+B^{2}\neq0$
on a nonempty open subset and the structure equations \eqref{struttura}
for equation \eqref{eq:C-2} have been reduced to the second identity
of \eqref{(7.9)*} together with \eqref{(7.3)*}-\eqref{(7.5)*},
one necessarily has $\mu_{2}^{2}+\mu_{3}^{2}\neq0$.

Thus, the second identity of \eqref{(7.9)*} now reads 
\begin{eqnarray*}
-\mu_{2}\,A+f_{22,z_{2}}=0,\\
-\mu_{3}\,A+f_{32,z_{2}}=0.
\end{eqnarray*}
and is equivalent to the system \eqref{A3D*}-\eqref{first.f22.f32*}
where $\varphi=\varphi(z,z_{1})$ is a differentiable function. Indeed,
\eqref{A3D*} is obtained by adding the first equation multiplied
by $\mu_{2}$ with the second one multiplied by $\mu_{3}$, whereas
\eqref{first.f22.f32*} readily follows by adding the first equation
multiplied by $\mu_{3}$ with the second equation multiplied by $-\mu_{2}$
and further integrating the resulting equation with respect to $z_{2}$.
Then, by substituting \eqref{first.f22.f32*} into \eqref{(7.3)*}
one also gets \eqref{sec.f22.f32*}.

Moreover, by multiplying \eqref{(7.4)*} by $\mu_{2}$ and adding
with \eqref{(7.5)*} multiplied by $\mu_{3}$ we obtain \eqref{B3D*}.
Analogously, by multiplying \eqref{(7.4)*} by $\mu_{3}$ and adding
with \eqref{(7.5)*} multiplied by $-\mu_{2}$, equation \eqref{first.f22.f32*}
leads to \eqref{last.eq.f22.f32*}. Finally, \eqref{7.5.1*} follows
by $\omega_{1}\wedge\omega_{2}\neq0$.

The converse is a straightforward computation. This concludes the
proof of Lemma \ref{lema1*}. 
\end{proof}
\begin{rem}
Although in the proof of Lemma \ref{lema1*} we assumed condition
\textbf{(ii)}, according to \eqref{fij*} one has that $f_{21}$ and
$f_{31}$ only depend on $z_{2}$. 
\end{rem}

\medskip{}

\subsection{Proof of Theorem \ref{teo.clas.3.2*}}

Following the notations used in Lemma \ref{lema1*}, we will consider
here $\left(f_{31,z_{2}}\right)^{2}-\delta\left(f_{21,z_{2}}\right)^{2}=\mu_{3}^{2}-\delta\mu_{2}^{2}=0$,
with either $\eta=0$ or $\eta\neq0$.

\subsubsection{Case $\eta=0$}

By Lemma \ref{lema1*} we need to determine the functions $f_{ij}$
that satisfy equations \eqref{fij*}-\eqref{7.5.1*}. First, in view
of $\mu_{3}^{2}-\delta\mu_{2}^{2}=0$, we immediately have $\delta=1$
and $\mu_{3}=\pm\mu_{2}$. Then condition $\mu_{3}^{2}+\mu_{2}^{2}\neq0$
(see Lemma \ref{lema1*}) leads to $\mu_{2}\neq0$ and, in view of
\eqref{first.f22.f32*}, one gets 
\begin{eqnarray}
f_{32}=\pm f_{22}-\dfrac{\varphi}{\mu_{2}},\label{g32.teo3.2*}
\end{eqnarray}
which replaced in \eqref{sec.f22.f32*} gives 
\begin{eqnarray}
(\alpha_{3}\mp\alpha_{2})f_{22}={\displaystyle \sum_{i=0}^{1}z_{i+1}f_{12,z_{i}}-z_{2}\,\varphi-\frac{\alpha_{2}}{\mu_{2}}\varphi.}\label{4.14*}
\end{eqnarray}
Furthermore, \eqref{last.eq.f22.f32*} is equivalent to the system
of the following two equations 
\begin{eqnarray}
\varphi_{,z_{1}}=0,\quad z_{1}\,\varphi_{,z}\pm\mu_{2}(\alpha_{3}\mp\alpha_{2})\,f_{12}=0.\label{4.15*}
\end{eqnarray}

From now on our analysis will be divided into the cases:\textbf{ (I)}
$\alpha_{3}\mp\alpha_{2}=0$;\textbf{ (II)} $\alpha_{3}\mp\alpha_{2}\neq0$.

\textbf{Case (I) $\boxed{\alpha_{3}\mp\alpha_{2}=0}$.} In this case
\eqref{4.15*} leads to $\varphi=m\,\mu_{2}$, $m\in\mathbb{R}$ ,
and \eqref{4.14*} is equivalent to $f_{12}=m\,\mu_{2}\,z_{1}+n$
and $m\,\alpha_{2}=0$, with $n\in\mathbb{R}$. In particular, by
\eqref{7.5.1*} one has $m^{2}+n^{2}\neq0$.

Thus, by replacing \eqref{g32.teo3.2*}, $\varphi=m\,\mu_{2}$ and
$f_{22}:=\mu_{2}\psi$, where $\psi=\psi(z,z_{1},z_{2})$ is a differentiable
function, in the equations \eqref{A3D*} and \eqref{B3D*} one gets
the partial differential equation 
\[
z_{2t}=D_{x}\psi\pm m\,\mu_{2}\,z_{1}z_{2}\pm\,n\left(z_{2}+\frac{\alpha_{2}}{\mu_{2}}\right)\,,\qquad m,n\in\mathbb{R},
\]
together with the corresponding functions 
\[
\begin{array}{ll}
 & f_{11}=0,\vspace*{0.3cm}\\
 & f_{21}=\mu_{2}\,z_{2}+\alpha_{2},\vspace*{0.3cm}\\
 & f_{31}=\pm\,\mu_{2}\,z_{2}\pm\alpha_{2},
\end{array}\hspace*{1cm}\begin{array}{ll}
f_{12}=m\,\mu_{2}\,z_{1}+n,\quad m^{2}+n^{2}\neq0,\vspace*{0.3cm}\\
f_{22}=\mu_{2}\,\psi,\vspace*{0.3cm}\\
f_{32}=\pm\,\mu_{2}\,\psi-m.
\end{array}
\]
Then, by performing a transformation of the form 
\[
z\mapsto\mu_{2}\,z,\qquad\psi\mapsto{\displaystyle \mu_{2}\,\psi\pm\frac{m\mu_{2}^{2}}{2}z_{1}^{2}\pm n\mu_{2}z_{1}}
\]
and further rearranging arbritrary functions, one gets the differential
equation given by (I) and associated $1$-forms, with $\eta=0$, $a=0$,
$b=0$ and $\alpha=\alpha_{2}$.

\medskip{}
 \textbf{Case (II) $\boxed{\alpha_{3}\mp\alpha_{2}\neq0}$.} In this
case, \eqref{4.15*} leads to $\varphi(z,z_{1})=h$, with $h=h(z)$
a differentiable function, and 
\begin{eqnarray}
f_{12}=\mp\dfrac{h'}{\mu_{2}(\alpha_{3}\mp\alpha_{2})}\,z_{1}.\label{4.16*}
\end{eqnarray}
In particular, from \eqref{7.5.1*} one gets that $h'\neq0$. Thus,
in view of \eqref{4.16*}, \eqref{4.14*},\eqref{g32.teo3.2*}, \eqref{A3D*}
and \eqref{B3D*} one gets the partial differential equation 
\begin{eqnarray*}
z_{2t}=-\dfrac{h}{\mu_{2}(\alpha_{3}\mp\alpha_{2})}z_{3}\mp\,D_{x}^{3}\left(\dfrac{h}{\mu_{2}^{2}(\alpha_{3}\mp\alpha_{2})^{2}}\right)-2\left((\alpha_{3}\mp\alpha_{2})\mu_{2}\,z_{2}\pm\dfrac{\alpha_{3}^{2}-\alpha_{2}^{2}}{2}\right)D_{x}\left(\dfrac{h}{\mu_{2}^{2}(\alpha_{3}\mp\alpha_{2})^{2}}\right).
\end{eqnarray*}
Furthermore, in view of \eqref{fij*}, \eqref{g32.teo3.2*}, \eqref{4.14*}
and \eqref{4.16*} one also gets the corresponding functions \vspace*{0.3cm}

\[
\begin{array}{ll}
 & f_{11}=0,\vspace*{0.3cm}\\
 & f_{21}=\mu_{2}z_{2}+\alpha_{2},\vspace*{0.3cm}\\
 & f_{31}=\pm\mu_{2}z_{2}+\alpha_{3},\vspace*{0.3cm}\\
 & f_{12}=\mp\dfrac{h'}{\mu_{2}(\alpha_{3}\mp\alpha_{2})}\,z_{1},\vspace*{0.3cm}\\
 & f_{22}=\dfrac{1}{\alpha_{3}\mp\alpha_{2}}\left(\mp\dfrac{h''}{\mu_{2}(\alpha_{3}\mp\alpha_{2})}\,z_{1}^{2}\mp\dfrac{h'}{\mu_{2}(\alpha_{3}\mp\alpha_{2})}\,z_{2}-z_{2}\,h-\dfrac{\alpha_{2}}{\mu_{2}}\,h\right),\vspace*{0.3cm}\\
 & f_{32}=\pm\dfrac{1}{\alpha_{3}\mp\alpha_{2}}\left(\mp\dfrac{h''}{\mu_{2}(\alpha_{3}\mp\alpha_{2})}\,z_{1}^{2}\mp\dfrac{h'}{\mu_{2}(\alpha_{3}\mp\alpha_{2})}\,z_{2}-z_{2}\,h\mp\dfrac{\alpha_{3}}{\mu_{2}}\,h\right).
\end{array}\hspace*{1cm}
\]

\vspace*{0.3cm}

\noindent Then, by performing a transformation of the form 
\[
\begin{array}{l}
t\mapsto\left(\alpha_{3}\mp\alpha_{2}\right)\mu_{2}\,t,\qquad z\mapsto\left(\alpha_{3}\mp\alpha_{2}\right)\mu_{2}\,z,\qquad h\mapsto{\displaystyle \frac{h}{\mu_{2}^{2}\left(\alpha_{3}\mp\alpha_{2}\right)^{2}}}\end{array}
\]
one gets the differential equation describing \textbf{pss} given in
(II) and associated $1$-forms, with $\eta=0$, $\alpha=\alpha_{3}$ and
$\beta=\alpha_{2}$.

The converse is a straightforward computation.


\subsubsection{Case $\eta\protect\neq0$}

By Lemma \ref{lema1*} we need to determine the functions $f_{ij}$
that satisfy equations (\ref{fij*})-(\ref{7.5.1*}). First, in view
of $\mu_{3}^{2}-\delta\mu_{2}^{2}=0$, we imediately have $\delta=1$
and $\mu_{3}=\pm\mu_{2}$. Then condition $\mu_{3}^{2}+\mu_{2}^{2}\neq0$
leads to $\mu_{2}\neq0$, and in view of \eqref{first.f22.f32*} and
\eqref{sec.f22.f32*}, one gets the following two conditions 
\begin{eqnarray}
 &  & f_{32}=\pm f_{22}-\dfrac{\varphi}{\mu_{2}},\label{4.23*}\\
 &  & (\alpha_{3}\mp\alpha_{2})\,f_{22}={\displaystyle \sum_{i=0}^{1}z_{i+1}\,f_{12,z_{i}}-z_{xx}\,\varphi-\dfrac{\alpha_{2}}{\mu_{2}}\,\varphi.\label{4.24*}}
\end{eqnarray}
Then, by substituting \eqref{4.23*} in \eqref{last.eq.f22.f32*}
one gets 
\begin{equation}
\mp\,\eta\,\varphi={\displaystyle \sum_{i=0}^{1}z_{i+1}\,\varphi_{,z_{i}}\pm\,\mu_{2}(\alpha_{3}\mp\alpha_{2})\,f_{12},}\label{4.25*}
\end{equation}

\noindent where, as shown above, $\varphi$ and $f_{12}$ are functions of $z$
and $z_{1}$. Thus, the derivative of \eqref{4.25*} with respect
to $z_{2}$ leads to $\varphi_{,z_{1}}=0$ and replacing back to \eqref{4.25*}
\begin{eqnarray}
\mp\,\eta\,h=z_{1}\,h_{,z}\pm\,\mu_{2}\,(\alpha_{3}\mp\alpha_{2})\,f_{12},\label{4.26*}
\end{eqnarray}
where $\varphi=h(z)$ is a differentiable function.

From now on our analysis will be divided into two cases:\textbf{ (I)}
$\alpha_{3}\mp\alpha_{2}=0$;\textbf{ (II)} $\alpha_{3}\mp\alpha_{2}\neq0$.

\textbf{Case (I) $\boxed{\alpha_{3}\mp\alpha_{2}=0}$.} In this case
\eqref{4.26*} leads to $h=0$ and \eqref{4.24*} to $f_{12}=m$ $\in$
$\mathbb{R}$. Then, by replacing \eqref{4.23*} and $f_{22}:=\mu_{2}\,\psi$,
with $\psi=\psi(z,z_{1},z_{2})$ being a differentiable function,
in \eqref{A3D*} and \eqref{B3D*}, one obtains the partial differential
equation 
\[
z_{2t}=D_{x}\psi\mp\eta\,\psi\pm\,m\left(\,z_{2}+\frac{\alpha_{2}}{\mu_{2}}\right),
\]
together with the corresponding functions 
\[
\begin{array}{ll}
 & f_{11}=\eta,\vspace*{0.4cm}\\
 & f_{21}=\mu_{2}\,z_{2}+{\displaystyle \alpha_{2}},\vspace*{0.4cm}\\
 & f_{31}=\pm\,\mu_{2}\,z_{2}\pm\alpha_{2},
\end{array}\hspace*{1cm}\begin{array}{ll}
f_{12}=m,\vspace*{0.4cm}\\
f_{22}=\mu_{2}\,\psi,\vspace*{0.4cm}\\
f_{32}=\pm\,\mu_{2}\,\psi.
\end{array}
\]
where $\eta\,\psi-m\,(\mu_{2}\,z_{2}+\alpha_{2})\neq0$ due to \eqref{7.5.1*}.

Then, by performing a transformation of the form 
\[
x\mapsto\mp\,\eta\,x,\qquad t\mapsto\pm\,m\,t,\;z\mapsto\mp{\displaystyle \eta\,\mu_{2}}\,z,\qquad\psi\mapsto\pm\left({\displaystyle \frac{\mu_{2}}{m}}\,\psi-\frac{\alpha_{2}}{\eta}\right),
\]
with $m\neq0$, and further rearranging arbitrary functions, we
get the differential equation describing \textbf{pss} given by (I)
and associated 1-forms, with $\eta\neq0$, $a=1$ and $b=1$.

Analogously, when $m=0$, by performing a transformation of the form
\[
x\mapsto\mp\,\eta\,x,\qquad z\mapsto\mp{\displaystyle \eta\,\mu_{2}}\,z,\qquad\psi\mapsto{\displaystyle \mu_{2}}\,\psi,
\]
one gets the differential equation describing \textbf{pss} given by
(I) and associated 1-forms, with $\eta\neq0$, $a=1$ and $b=0$.

\medskip{}

\textbf{Case (II) $\boxed{\alpha_{3}\mp\alpha_{2}\neq0}$.} In this
case, in view of \eqref{4.26*}, \eqref{4.24*}, \eqref{4.23*}, \eqref{A3D*}
and \eqref{B3D*} one gets the differential equation

\begin{eqnarray*}
z_{2t}=-\dfrac{h}{\mu_{2}(\alpha_{3}\mp\alpha_{2})}z_{3}\mp\,D_{x}^{3}\left(\dfrac{h}{\mu_{2}^{2}(\alpha_{3}\mp\alpha_{2})^{2}}\right)-2\left((\alpha_{3}\mp\alpha_{2})\mu_{2}\,z_{2}\pm\dfrac{\alpha_{3}^{2}-\alpha_{2}^{2}}{2}\mp\dfrac{\eta^{2}}{2}\right)D_{x}\left(\dfrac{h}{\mu_{2}^{2}(\alpha_{3}\mp\alpha_{2})^{2}}\right),
\end{eqnarray*}
together with the corresponding functions 
\[
\begin{array}{ll}
 & f_{11}=\eta,\vspace*{0.3cm}\\
 & f_{21}=\mu_{2}\,z_{2}+\alpha_{2},\vspace*{0.4cm}\\
 & f_{31}=\pm\mu_{2}\,z_{2}+\alpha_{3},\vspace*{0.4cm}\\
 & f_{12}=\mp\dfrac{h'}{\mu_{2}(\alpha_{3}\mp\alpha_{2})}\,z_{1}-\dfrac{\eta\,h}{\mu_{2}(\alpha_{3}\mp\alpha_{2})},\vspace*{0.4cm}\\
 & f_{22}=\dfrac{1}{(\alpha_{3}\mp\alpha_{2})}\left(\mp\dfrac{h''}{\mu_{2}(\alpha_{3}\mp\alpha_{2})}\,z_{1}^{2}\mp\dfrac{h'}{\mu_{2}(\alpha_{3}\mp\alpha_{2})}\,z_{2}-\dfrac{\eta\,h'}{\mu_{2}(\alpha_{3}\mp\alpha_{2})}\,z_{1}-z_{2}\,h-\dfrac{\alpha_{2}}{\mu_{2}}\,h\right),\vspace*{0.4cm}\\
 & f_{32}=\pm\dfrac{1}{(\alpha_{3}\mp\alpha_{2})}\left(\mp\dfrac{h''}{\mu_{2}(\alpha_{3}\mp\alpha_{2})}\,z_{1}^{2}\mp\dfrac{h'}{\mu_{2}(\alpha_{3}\mp\alpha_{2})}\,z_{2}-\dfrac{\eta\,h'}{\mu_{2}(\alpha_{3}\mp\alpha_{2})}\,z_{1}-z_{2}\,h\mp\dfrac{\alpha_{3}}{\mu_{2}}\,h\right),
\end{array}\hspace*{1cm}
\]
with $h'\neq0$ due to \eqref{7.5.1*}.

Then, by performing a transformation of the form 
\[
\begin{array}{l}
t\mapsto\left(\alpha_{3}\mp\alpha_{2}\right)\mu_{2}\,t,\qquad z\mapsto\left(\alpha_{3}\mp\alpha_{2}\right)\mu_{2}\,z,\qquad h\mapsto{\displaystyle \frac{h}{\mu_{2}^{2}\left(\alpha_{3}\mp\alpha_{2}\right)^{2}}}\end{array}
\]
one gets the differential equation describing \textbf{pss} given by
(II) and associated $1$-forms, with $\alpha=\alpha_{3}$ and $\beta=\alpha_{2}$.

The converse is a straightforward computation.


\subsection{Proof of Theorem \ref{teo.clas.3.3*}}

Following the notations used in Lemma \ref{lema1*}, we will consider
here $\left(f_{31,z_{2}}\right)^{2}-\delta\left(f_{21,z_{2}}\right)^{2}=\mu_{3}^{2}-\delta\mu_{2}^{2}\neq0$,
with $\eta=0$. Thus we need to determine the functions $f_{ij}$
that satisfy (\ref{fij*})-(\ref{7.5.1*}). First, from equation \eqref{last.eq.f22.f32*}
one gets the following two equations 
\begin{eqnarray}
f_{12}=-\dfrac{1}{\gamma}\,\varphi_{,z_{1}},\label{4.17*}
\end{eqnarray}

\begin{eqnarray}
m_{1}\,\varphi_{,z_{1}}-\gamma\,z_{1}\,\varphi_{,z}=0,\label{4.18*}
\end{eqnarray}
where 
\begin{equation}
\gamma:=\mu_{3}^{2}-\delta\,\mu_{2}^{2},\qquad m_{1}:=\mu_{3}\alpha_{3}-\delta\mu_{2}\alpha_{2}.\label{gamma_m1_thm3.2}
\end{equation}

One can distinguish two cases:\textbf{ (I)} $-\mu_{3}\alpha_{2}+\mu_{2}\alpha_{3}=0$;\textbf{
(II)} $-\mu_{3}\alpha_{2}+\mu_{2}\alpha_{3}\neq0$.

\textbf{Case (I) $\boxed{-\mu_{3}\alpha_{2}+\mu_{2}\alpha_{3}=0}$.}
In this case $\mu_{2}\neq0$, otherwise by $\mu_{2}^{2}+\mu_{3}^{2}\neq0$
one would get $\alpha_{2}=0$ which contradicts \eqref{7.5.1*}. Thus
$\alpha_{3}=\mu_{3}\alpha_{2}/\mu_{2}$ and by \eqref{first.f22.f32*}
and \eqref{sec.f22.f32*} one gets 
\begin{eqnarray}
\dfrac{\alpha_{2}}{\mu_{2}}\varphi=\dfrac{\alpha_{2}}{\mu_{2}}\left(\mu_{3}f_{22}-\mu_{2}f_{32}\right)=\alpha_{3}\,f_{22}-\alpha_{2}\,f_{32} & = & {\displaystyle \sum_{i=0}^{1}z_{i+1}f_{12,z_{i}}-z_{2}\,\varphi.}\label{eq_aux}
\end{eqnarray}
Hence, by differentiating \eqref{eq_aux} with respect to $z_{2}$,
one has 
\begin{eqnarray}
f_{12,z_{1}}-\varphi=0,\qquad f_{12,z}-\dfrac{\alpha_{2}}{\mu_{2}}\dfrac{\varphi}{z_{1}}=0\label{4.19*}
\end{eqnarray}
and the compatibility condition $f_{12,z\,z_{1}}=f_{12,z_{1}\,z}$
reads 
\[
z_{1}\,\left[\frac{\alpha_{2}}{\mu_{2}}\,\varphi_{,z_{1}}-z_{1}\,\varphi_{,z}\right]-\frac{\alpha_{2}}{\sigma_{2}}\,\varphi=0.
\]
On the other hand $\alpha_{3}=\mu_{3}\alpha_{2}/\mu_{2}$ entails
that $m_{1}=\alpha_{2}\,\gamma/\mu_{2}$, and \eqref{4.18*} leads
to 
\begin{eqnarray}
\frac{\alpha_{2}}{\mu_{2}}\,\varphi_{,z_{1}}-z_{1}\,\varphi_{,z}=0\label{4.20*}
\end{eqnarray}
and $\alpha_{2}\,\varphi=0$. Now, if $\varphi=0$ then $f_{12}=0$
due to \eqref{4.17*}, which contradicts \eqref{7.5.1*}. It follows
that $\alpha_{2}=0$, hence $\alpha_{3}=\mu_{3}\alpha_{2}/\mu_{2}=0$
and \eqref{4.20*} entails that $\varphi=\phi(z_{1})$. On the other
hand, due to \eqref{4.17*} and \eqref{4.19*}, $\phi$ satisfy $\phi_{,z_{1}z_{1}}+\gamma\,\phi=0$
and so can be written as 
\begin{equation}
\phi=\left\{ \begin{array}{lc}
a\,\cos\left(\sqrt{\gamma}\,z_{1}\right)+b\,\sin\left(\sqrt{\gamma}\,z_{1}\right),\vspace{5pt}\qquad & \gamma>0,\\
a\,\cosh\left(\sqrt{-\gamma}\,z_{1}\right)+b\,\sinh\left(\sqrt{-\gamma}\,z_{1}\right),\qquad & \gamma<0,
\end{array}\right.\label{eq:cosh_cos_etc}
\end{equation}
with $a,b\in\mathbb{R}$, $a^{2}+b^{2}\neq0$. In particular, for
\eqref{gamma_m1_thm3.2}, $\gamma<0$ implies $\delta=1$.

Then, by replacing \eqref{4.17*}, \eqref{first.f22.f32*} and $f_{22}=\mu_{2}\,\psi$,
with $\psi=\psi(z,z_{1},z_{2})$ a differentiable function, in \eqref{A3D*}
and \eqref{B3D*}, one gets the partial differential equation 
\begin{eqnarray}
z_{2t}=D_{x}\psi-\gamma\,\mu_{2}\,\mu_{3}\,\phi'\,z_{2},
\end{eqnarray}
with 
\[
\begin{array}{ll}
 & f_{11}=0,\vspace*{0.4cm}\\
 & f_{21}=\pm\sqrt{\delta\left(\mu_{3}^{2}-\gamma\right)}\,z_{2},\vspace*{0.4cm}\\
 & f_{31}=\mu_{3}\,z_{2},
\end{array}\hspace*{1cm}\begin{array}{ll}
 & f_{12}=\mp\sqrt{\delta\left(\mu_{3}^{2}-\gamma\right)}\,\mu_{2}\,\gamma\,\phi',\vspace*{0.4cm}\\
 & f_{22}=\pm\sqrt{\delta\left(\mu_{3}^{2}-\gamma\right)}\,\psi,\vspace*{0.4cm}\\
 & f_{32}=\mu_{3}\,\psi-\gamma^{2}\mu_{2}\,\phi,
\end{array}
\]
and $\phi(z_{1})$ given by \eqref{eq:cosh_cos_etc}.

Finally, by performing a transformation of the form 
\[
\begin{array}{l}
t\mapsto\mu_{2}\,\gamma\,t,\qquad z\mapsto\pm\sqrt{\delta\left(\mu_{3}^{2}-\gamma\right)}\,z,\qquad\psi\mapsto\pm{\displaystyle \frac{\sqrt{\delta\left(\mu_{3}^{2}-\gamma\right)}}{\mu_{2}\gamma}\,\left(\psi-\mu_{2}\mu_{3}\gamma\,\phi\right),\qquad\phi\mapsto\delta\left(\mu_{3}^{2}-\gamma\right)\phi}\end{array},
\]
then passing to $r:=\mu_{3}/\sqrt{\delta\left(\mu_{3}^{2}-\gamma\right)}$
and rearranging the constants $a,b$ in the expression of $\phi$
provided by \eqref{eq:cosh_cos_etc}, one gets the differential equation
describing \textbf{pss} and \textbf{ss} given by (I), together with
the associated $1$-forms.

\medskip{}

\textbf{Case (II) $\boxed{-\mu_{3}\alpha_{2}+\mu_{2}\alpha_{3}\neq0}$.}
In this case \eqref{first.f22.f32*} and \eqref{sec.f22.f32*} can
be solved with respect to $f_{22}$ and $f_{32}$ as 
\begin{eqnarray}
 &  & f_{22}=\frac{1}{m_{2}}\left[-\alpha_{2}\,\varphi+\mu_{2}\,{\displaystyle \sum_{i=0}^{1}z_{i+1}f_{12,z_{i}}-\mu_{2}\,z_{2}\,\varphi}\right],\label{4.21*}\\
 &  & f_{32}=\frac{1}{m_{2}}\left[-\alpha_{3}\,\varphi+\mu_{3}\,{\displaystyle \sum_{i=0}^{1}z_{i+1}f_{12,z_{i}}-\mu_{3}\,z_{2}\,\varphi}\right],\label{4.22*}
\end{eqnarray}
where $f_{12}$ is given by \eqref{4.17*}, $\varphi$ satisfies \eqref{4.18*}
and 
\begin{equation}
m_{2}:=\mu_{2}\alpha_{3}-\mu_{3}\alpha_{2}.\label{eq:m2_th3.5}
\end{equation}

Hence, by substituting \eqref{4.21*} and \eqref{4.22*} in \eqref{A3D*}
and \eqref{B3D*}, and using \eqref{4.17*} and \eqref{4.18*}, we
get the partial differential equation 
\begin{center}
\begin{eqnarray}
z_{2t}=-D_{x}^{2}\left(\dfrac{\varphi_{,z_{1}}}{\gamma}\right)-D_{x}\left[\left(z_{2}+\dfrac{m_{1}}{\gamma}\right)\,\varphi\right]+\frac{\delta\,m_{2}^{2}}{\gamma^{2}}\,\varphi_{,z_{1}},\label{eq..iiiii3*}
\end{eqnarray}
\par\end{center}

\vspace*{0.3cm}
\noindent where $m_{2}\gamma\neq0$ and, in order to simplify the obtained
partial differential equation, $\varphi$ was rearranged as $\varphi\mapsto\varphi/m_{2}$.
Accordingly, in view of \eqref{4.17*} and \eqref{4.21*}-\eqref{4.22*}
one gets the corresponding functions \\

\noindent 
\[
\begin{array}{ll}
 & f_{11}=0,\vspace*{0.3cm}\\
 & f_{21}=\mu\,z_{2}+\dfrac{\mu\,m_{1}\mp\,m_{2}\,\sqrt{\gamma+\delta\mu^{2}}}{\gamma},\qquad\gamma+\delta\mu^{2}\geq0,\qquad\mu\in\mathbb{R},\vspace*{0.3cm}\\
 & f_{31}=\pm\sqrt{\gamma+\delta\mu^{2}}\,\,z_{2}-\dfrac{\delta\mu\,m_{2}\mp\,m_{1}\sqrt{\gamma+\delta\mu^{2}}}{\gamma},\vspace*{0.3cm}\\
 & f_{12}=-\dfrac{m_{2}}{\gamma}\,\varphi_{,z_{1}}\,,\vspace*{0.3cm}\\
 & f_{22}=-\mu\,D_{x}\left(\dfrac{\varphi_{,z_{1}}}{\gamma}\right)-\mu\,z_{2}\,\varphi-\dfrac{\mu\,m_{1}\mp\,m_{2}\,\sqrt{\gamma+\delta\mu^{2}}}{\gamma}\,\varphi\,,\vspace*{0.3cm}\\
 & f_{32}=\mp\,\sqrt{\gamma+\delta\mu^{2}}\,\,D_{x}\left(\dfrac{\varphi_{,z_{1}}}{\gamma}\right)\mp\,\sqrt{\gamma+\delta\mu^{2}}\,\,z_{2}\,\varphi+\dfrac{\delta\mu\,m_{2}\mp\,m_{1}\sqrt{\gamma+\delta\mu^{2}}}{\gamma}\,\varphi\,,
\end{array}\hspace*{1cm}
\]

\vspace*{0.3cm}
\noindent where $\mu:=\mu_{2}$, $\gamma+\delta\,\mu^{2}\geq0$, $\varphi=\varphi(z,z_{1})$
is a differentiable function satisfying $\varphi_{,z}\neq0$ and \eqref{4.18*};
moreover we took account of \eqref{gamma_m1_thm3.2} and \eqref{eq:m2_th3.5}
by setting $\mu_{3}=\pm\,\sqrt{\gamma+\delta\,\mu^{2}}$, $\alpha_{2}=(\mu\,m_{1}-\mu_{3}\,m_{2})/\gamma$
and $\alpha_{3}=(\mu_{3}\,m_{1}-\delta\,\mu\,m_{2})/\gamma$. Now,
by performing the transformation $z\mapsto\gamma z$, one can see
that \eqref{4.18*} reduces to $m_{1}\varphi_{,z_{1}}-z_{1}\varphi_{,z}=0$.
Hence, without loss of generality one can take $\varphi=\varphi(z_{1}^{2}+2\,m_{1}\,z)$
and, by introducing the constants $r:=m_{2}$ and $m:=m_{1}$, one finally
gets the partial differential equation describing \textbf{pss} or
\textbf{ss} given in (II), together with the associated $1$-forms.

The converse is a straightforward computation.


\subsection{Proof of Theorem \ref{teo.clas.3.5*}}

Following the notations used in Lemma \ref{lema1*}, we will consider
here $\left(f_{31,z_{2}}\right)^{2}-\delta\left(f_{21,z_{2}}\right)^{2}=\mu_{3}^{2}-\delta\mu_{2}^{2}\neq0$,
with $\eta\neq0$. Thus we need to determine the functions $f_{ij}$
that satisfy (\ref{fij*})-(\ref{7.5.1*}). First, from \eqref{first.f22.f32*}
and \eqref{last.eq.f22.f32*} one gets 
\begin{eqnarray}
 &  & f_{22}=\dfrac{\mu_{3}}{\gamma}\,\varphi+\dfrac{\mu_{2}}{\eta\gamma}{\displaystyle \sum_{i=0}^{1}z_{i+1}\,\varphi_{,z_{i}}+\dfrac{\mu_{2}}{\eta\gamma}(\gamma\,z_{2}+m_{1})\,f_{12}\,,\label{4.27*}}\\
 &  & f_{32}=\dfrac{\delta\mu_{2}}{\gamma}\,\varphi+\dfrac{\mu_{3}}{\eta\gamma}{\displaystyle \sum_{i=0}^{1}z_{i+1}\,\varphi_{,z_{i}}+\dfrac{\mu_{3}}{\eta\gamma}(\gamma\,z_{2}+m_{1})\,f_{12}\,,\label{4.28*}}
\end{eqnarray}
where $\varphi^{2}+f_{12}^{2}\neq0$, in view of \eqref{7.5.1*},
and $\gamma$, $m_{1}$ are the constants defined in \eqref{gamma_m1_thm3.2}.

Then, by substituting \eqref{4.27*} and \eqref{4.28*} in \eqref{sec.f22.f32*},
and differentiating with respect to $z_{2}$ one get the following
two equations 
\begin{eqnarray}
\ell_{1,z_{1}}-\,\gamma\,\ell_{2}=0,\label{4.29*}\\
\nonumber \\
z_{1}\,\ell_{1,z}-m_{1}\,\ell_{2}=0,\label{4.30*}
\end{eqnarray}
where 
\begin{equation}
\ell_{1}:=\eta\,\gamma\,f_{12}-m_{2}\varphi,\qquad\ell_{2}:=m_{2}\,f_{12}+\eta\,\varphi\label{eq:l1-l2-thm3.6*}
\end{equation}
and $m_{2}$ is the constant defined in \eqref{eq:m2_th3.5}.

Now, in view of $\gamma\neq0$, by \eqref{4.29*} one has $\ell_{2}=\ell_{1,z_{1}}/\gamma$
and \eqref{4.30*} leads to 
\begin{eqnarray}
m_{1}\,\ell_{1,z_{1}}-\gamma\,z_{1}\,\ell_{1,z}=0.\label{4.31*}
\end{eqnarray}
On the other hand, by using the definition of $\ell_{2}$ and \eqref{4.29*},
one obtains 
\begin{eqnarray}
\varphi=-\dfrac{m_{2}}{\eta}\,f_{12}+\dfrac{\ell_{1,z_{1}}}{\eta\,\gamma},\label{4.32*}
\end{eqnarray}
which replaced in the definition of $\ell_{1}$ leads to 
\begin{eqnarray}
(\gamma\,\eta^{2}+m_{2}^{2})\,f_{12}=\eta\,\ell_{1}+\frac{m_{2}}{\gamma}\,\ell_{1,z_{1}}.\label{4.33*}
\end{eqnarray}
From now on we will distinguish the cases:\textbf{ (I)} $\gamma\,\eta^{2}+m_{2}^{2}=0$;\textbf{
(II)} $\gamma\,\eta^{2}+m_{2}^{2}\neq0$.

\textbf{Case (I) $\boxed{\gamma\,\eta^{2}+m_{2}^{2}=0}$. }First we
notice that in this case $m_{2}\neq0$ and by \eqref{4.33*} 
\[
\ell_{1,z_{1}}+\dfrac{\eta\,\gamma}{m_{2}}\ell_{1}=0,
\]
that is $\ell_{1}(z,z_{1})=r\,e^{-\frac{\eta\gamma}{m_{2}}z_{1}}$,
with $m_{1}=0$ and $r\in\mathbb{R}$, due to \eqref{4.31*}. Now,
by \eqref{4.32*} 
\begin{eqnarray}
f_{12}=-\dfrac{\eta}{m_{2}}\varphi-\dfrac{r\,\eta}{m_{2}^{2}}\,e^{-\frac{\eta\gamma}{m_{2}}z_{1}}\,,\label{f12.*}
\end{eqnarray}
and in view of \eqref{4.27*}-\eqref{4.28*} one obtains $f_{22}$
and $f_{32}$ in terms of $\varphi$. Then, in view of \eqref{fij*},
\eqref{A3D*}, \eqref{B3D*}, \eqref{4.27*}, \eqref{4.28*} and \eqref{f12.*},
one obtains the following partial differential equation

\begin{eqnarray}
z_{2t}=\dfrac{1}{\eta\,\gamma}\,D_{x}^{2}\left(\varphi\right)-D_{x}\left(\dfrac{\varphi}{m_{2}}\,z_{2}\right)-\dfrac{r}{m_{2}^{2}}\,D_{x}(z_{2}\,e^{-\frac{\eta\,\gamma}{m_{2}}\,z_{1}})+\dfrac{r\,\eta}{m_{2}\,\gamma}\,e^{-\frac{\eta\,\gamma}{m_{2}}\,z_{1}},\label{eq55.t3*}
\end{eqnarray}
together with the corresponding functions 
\[
\begin{array}{ll}
 & f_{11}=\eta,\vspace*{0.4cm}\\
 & f_{21}=\mu\,z_{2}\mp\,\dfrac{m_{2}\,\sqrt{\gamma+\mu^{2}}}{\gamma},\qquad\gamma+\mu^{2}\geq0,\qquad\mu\in\mathbb{R}\setminus\{0\},\vspace*{0.4cm}\\
 & f_{31}=\pm\sqrt{\gamma+\mu^{2}}\,\,z_{2}-\dfrac{\mu\,m_{2}}{\gamma},\vspace*{0.4cm}\\
 & f_{12}=-\dfrac{\eta}{m_{2}}\varphi-\dfrac{r\,\eta}{m_{2}^{2}}\,e^{-\frac{\eta\gamma}{m_{2}}z_{1}},\vspace*{0.4cm}\\
 & f_{22}=\pm\dfrac{\sqrt{\gamma+\mu^{2}}}{\gamma}\,\varphi+\dfrac{\mu}{\eta\gamma}\,D_{x}(\varphi)+\dfrac{\mu}{\eta}\,z_{2}\,f_{12}\,,\vspace*{0.4cm}\\
 & f_{32}=\dfrac{\mu}{\gamma}\,\varphi\pm\dfrac{\sqrt{\gamma+\mu^{2}}}{\eta\gamma}\,D_{x}(\varphi)\pm\dfrac{\sqrt{\gamma+\mu^{2}}}{\eta}\,z_{2}\,f_{12}\,,
\end{array}\hspace*{1cm}
\]

\vspace*{0.3cm}
\noindent where $\mu:=\mu_{2}$, $m_{2}\,\gamma\neq0$, $\eta^{2}\,\gamma+m_{2}^{2}=0$
and $\varphi=\varphi(z,z_{1})$ is a differentiable function satisfying
$\mu(\varphi_{,z}\,z_{1}+\varphi_{,z_{1}}\,z_{2})\neq0$, due to \eqref{7.5.1*}.
In particular, we set $\mu_{3}=\pm\sqrt{\gamma+\mu^{2}}$, $\alpha_{2}=-m_{2}\mu_{3}/\gamma$
and $\alpha_{3}=-m_{2}\mu_{2}/\gamma$, in view of \eqref{gamma_m1_thm3.2}
and \eqref{eq:m2_th3.5}. 
 
Now, by performing a transformation 

\[
z\mapsto\frac{\eta\gamma}{m_{2}}z,\qquad\varphi\mapsto\frac{1}{m_{2}}\left(\varphi+\frac{r}{m_{2}}e^{-\frac{\eta\gamma}{m_{2}}z_{1}}\right),
\]
rearranging $r$ according to $r\mapsto\left(\eta^{2}/m_{2}^{2}\right)r$
and introducing $\rho:=\eta\mu/m_{2}$, one gets the partial differential
equation describing \textbf{pss} given in (I), together with the associated
$1$-forms.
\par

\textbf{Case (II) $\boxed{\gamma\,\eta^{2}+m_{2}^{2}\neq0}$.} From
\eqref{4.33*} and \eqref{4.32*} one obtains $f_{12}$ and $\varphi$
in tems of $\ell_{1}$ and $\ell_{1,z_{1}}$, where $\ell_{1}$ is
solution of \eqref{4.31*}. Then, by substituting \eqref{4.27*} and
\eqref{4.28*} in \eqref{A3D*} and \eqref{B3D*}, one has the following
partial differential equation 
\begin{eqnarray}
z_{2t}=\dfrac{1}{\eta\,\gamma}D_{x}^{2}(\varphi)+\frac{1}{\eta\,\gamma}D_{x}\left[\phi\,(\gamma\,z_{2}+m_{1})\right]-\dfrac{\delta\,m_{2}}{\gamma}\,\phi-\dfrac{\delta\,\eta\,\varphi}{\gamma},\label{eq56_prof*}
\end{eqnarray}
\noindent where $\varphi=\varphi(z,z_{1})$ and $\phi=\phi(z,z_{1})$ are 
\begin{eqnarray*}
 &  & \varphi=\dfrac{\eta\,\ell_{,z_{1}}}{\gamma\,\eta^{2}+m_{2}^{2}}-\dfrac{m_{2}\,\ell}{\gamma\,\eta^{2}+m_{2}^{2}}\,,\\
\\
 &  & \phi=\dfrac{m_{2}\,\ell_{,z_{1}}}{\gamma\,(\gamma\,\eta^{2}+m_{2}^{2})}+\dfrac{\eta\,\ell}{\gamma\,\eta^{2}+m_{2}^{2}}\,,\\
\end{eqnarray*}
and $\gamma(\gamma\,\eta^{2}+m_{2}^{2})\neq0$. Also, from \eqref{fij*},
\eqref{4.27*}, \eqref{4.28*}, \eqref{4.31*}, \eqref{4.32*} and
\eqref{4.33*}, one gets the corresponding functions

\[
\begin{array}{ll}
 & f_{11}=\eta,\vspace*{0.3cm}\\
 & f_{21}=\mu\,z_{2}+\dfrac{\mu\,m_{1}\mp\,m_{2}\,\sqrt{\gamma+\delta\mu^{2}}}{\gamma},\qquad\gamma+\delta\mu^{2}\geq0,\qquad\mu\in\mathbb{R},\vspace*{0.4cm}\\
 & f_{31}=\pm\sqrt{\gamma+\delta\mu^{2}}\,z_{2}-\dfrac{\delta\mu\,m_{2}\mp\,m_{1}\sqrt{\gamma+\delta\mu^{2}}}{\gamma},\vspace*{0.4cm}\\
 & f_{12}=\dfrac{m_{2}\,\ell_{,z_{1}}}{\gamma\,(\gamma\,\eta^{2}+m_{2}^{2})}+\dfrac{\eta\,\ell}{\gamma\,\eta^{2}+m_{2}^{2}}\,,\vspace*{0.4cm}\\
 & f_{22}=\pm\dfrac{\sqrt{\gamma+\delta\mu^{2}}}{\gamma}\,\varphi+\dfrac{\mu}{\eta\gamma}D_{x}(\varphi)+\dfrac{\mu}{\eta\gamma}\left(\gamma\,z_{2}+m_{1}\right)\,\phi\,,\vspace*{0.4cm}\\
 & f_{32}=\dfrac{\delta\mu}{\gamma}\,\varphi\pm\dfrac{\sqrt{\gamma+\delta\mu^{2}}}{\eta\gamma}D_{x}(\varphi)\pm\dfrac{\sqrt{\gamma+\delta\mu^{2}}}{\eta\gamma}\left(\gamma\,z_{2}+m_{1}\right)\,\phi\,,
\end{array}\hspace*{1cm}
\]

\vspace*{0.3cm}
\noindent where $\mu:=\mu_{2}$, $\gamma+\delta\,\mu^{2}\geq0$, $\ell:=\ell_{1}$
is a nonzero differentiable function satisfying \eqref{4.31*}. In
particular, we took account of \eqref{gamma_m1_thm3.2} and \eqref{eq:m2_th3.5}
by setting $\mu_{3}=\pm\,\sqrt{\gamma+\delta\,\sigma^{2}}$, $\alpha_{2}=(\mu\,m_{1}-\mu_{3}\,m_{2})/\gamma$
and $\alpha_{3}=(\mu_{3}\,m_{1}-\delta\,\mu\,m_{2})/\gamma$. In addition,
$\eta\,f_{22}-\phi\,f_{21}\neq0$ due to nondegeneracy condition \eqref{7.5.1*}.

 Now, by performing a transformation $z\mapsto\gamma z$,
equation \eqref{4.31*} reduces to $m_{1}\ell_{,z_{1}}-z_{1}\ell_{,z}=0$.
In particular, under such a transformation $\ell'\mapsto\ell'/\gamma^{2}$.
Hence without loss of generality one can take $\ell=\ell(z_{1}^{2}+2\,m_{1}\,z)$
and, by introducing the constants $r:=m_{2},m:=m_{1},\mu:=\mu_{2}/\gamma$ and $\mu_{3}=\sqrt{\gamma+\delta\,\mu^{2}}$,
one finally gets the partial differential equation describing \textbf{pss}
or \textbf{ss} given in (II), together with the associated $1$-forms.
In particular, $\ell'\neq0$ due to $\omega_{1}\wedge\omega_{2}\neq0$.

The converse is a straightforward computation.

\end{document}